\documentclass[a4paper]{article}
\pdfoutput=1
\usepackage{hyperref}
\hypersetup{hypertexnames = false, bookmarksdepth = 2, bookmarksopen = true, colorlinks, linkcolor = black, citecolor = black, urlcolor = black, pdfstartview={XYZ null null 1}}

\usepackage{amsfonts}
\usepackage[fleqn, leqno]{amsmath}
\usepackage{amsthm}

\usepackage[
  margin=3.5cm,
  includefoot,
  footskip=20pt,
]{geometry}
\usepackage{layout}

\usepackage[capitalise]{cleveref}

\usepackage[backend=biber, maxbibnames=99, datamodel=mrnumber, sortcites, backref=true]{biblatex}

\usepackage{adjustbox}
\usepackage{booktabs}
\usepackage{colonequals}
\usepackage{diagbox}
\usepackage{enumitem}
\usepackage{mathtools}
\usepackage{parskip}
\usepackage{rotating}
\usepackage{thmtools}
\usepackage{tikz-cd}
\usepackage[colorinlistoftodos, textsize = footnotesize]{todonotes}
\usepackage{xparse}
\usepackage{xspace}
\usepackage{ytableau}

\usepackage[utf8]{inputenc}
\usepackage[T1]{fontenc}
\usepackage{libertine}
\usepackage[libertine]{newtxmath}
\usepackage[scaled=0.83]{beramono}
\usepackage{eucal}
\usepackage{microtype}
\usepackage{gitinfo2}

\DeclareUnicodeCharacter{0306}{}

\DeclareFieldFormat{mrnumber}{%
  MR\addcolon\space
  \ifhyperref
    {\href{http://www.ams.org/mathscinet-getitem?mr=#1}{\nolinkurl{#1}}}
    {\nolinkurl{#1}}}

\renewbibmacro*{doi+eprint+url}{%
  \iftoggle{bbx:doi}
    {\printfield{doi}}
    {}%
  \newunit\newblock
  \printfield{mrnumber}%
  \newunit\newblock
  \iftoggle{bbx:eprint}
    {\usebibmacro{eprint}}
    {}%
  \newunit\newblock
  \iftoggle{bbx:url}
    {\usebibmacro{url+urldate}}
    {}}

\DefineBibliographyStrings{english}{
  backrefpage={cited on page},
  backrefpages={cited on pages}
}

\usetikzlibrary{arrows.meta}

\newcounter{todocounter}
\DeclareDocumentCommand\addreference{g}{\stepcounter{todocounter}\todo[color = blue!30]{\thetodocounter. Add reference\IfNoValueF{#1}{: #1}}\xspace}
\DeclareDocumentCommand\checkthis{g}{\stepcounter{todocounter}\todo[color = red!50]{\thetodocounter. Check this\IfNoValueF{#1}{: #1}}\xspace}
\DeclareDocumentCommand\fixthis{g}{\stepcounter{todocounter}\todo[color = orange!50]{\thetodocounter. Fix this\IfNoValueF{#1}{: #1}}\xspace}
\DeclareDocumentCommand\expand{g}{\stepcounter{todocounter}\todo[color = green!50]{\thetodocounter. Expand\IfNoValueF{#1}{: #1}}\xspace}

\newtheorem{theorem}{Theorem}
\newtheorem{corollary}[theorem]{Corollary}
\newtheorem{lemma}[theorem]{Lemma}
\newtheorem{proposition}[theorem]{Proposition}
\declaretheoremstyle[
  spaceabove = 3pt,
  spacebelow = 3pt,
]{lecture}
\theoremstyle{lecture}
\newtheorem{definition}[theorem]{Definition}
\newtheorem{example}[theorem]{Example}

\newtheorem{remark}[theorem]{Remark}
\newtheorem{setup}[theorem]{Setup}

\makeatletter
\def\gitfootnote{\gdef\@thefnmark{}\@footnotetext}
\makeatother

\ytableausetup{boxsize=.5em,aligntableaux=center}

\mathchardef\mhyphen="2D
\newcommand\dash{\nobreakdash-\hspace{0pt}}
\renewcommand\coloneqq\colonequals

\newcommand{\sslash}{\mathbin{/\mkern-6mu/}}

\DeclareMathOperator\Aut{Aut}
\DeclareMathOperator\Bl{Bl}
\DeclareMathOperator\derived{\mathbf{D}}
\DeclareMathOperator\Ext{Ext}
\DeclareMathOperator\Gr{Gr}
\DeclareMathOperator\HH{H}
\DeclareMathOperator\HHHH{HH}
\DeclareMathOperator\Hom{Hom}
\DeclareMathOperator\Ind{Ind}
\DeclareMathOperator\rk{rk}
\DeclareMathOperator\Spec{Spec}
\DeclareMathOperator\Sym{Sym}

\newcommand\bounded{\ensuremath{\mathrm{b}}}
\newcommand\LLL{\ensuremath{\mathbf{L}}}
\newcommand\RRR{\ensuremath{\mathbf{R}}}
\newcommand\subspace{\ensuremath{\mathrm{S}}}
\newcommand\sym{\ensuremath{\mathfrak{S}}}
\newcommand\tangent{\ensuremath{\mathrm{T}}}

\newcommand\sing{\ensuremath{\mathrm{sing}}}
\newcommand\coble[1][]{\ensuremath{\vphantom{\mathsf{Cob}}\smash{\mathsf{Cob}^{#1}}}}
\newcommand\igusa[1][]{\ensuremath{\vphantom{\mathsf{CR}_4}\smash{\mathsf{CR}_4^{#1}}}}
\newcommand\segre[1][]{\ensuremath{\vphantom{\mathsf{S}_3}\smash{\mathsf{S}_3^{#1}}}}

\title{Homological projective duality for the Segre~cubic}
\author{Thorsten Beckmann, Pieter Belmans}

\begin{document}
\maketitle


\begin{abstract}
  The Segre cubic and Castelnuovo--Richmond quartic are two projectively dual hypersurfaces in $\mathbb{P}^4$,
  with a long and rich history starting in the 19th century.
  We will explain how Kuznetsov's theory of homological projective duality lifts this projective duality
  to a relationship between the derived category of a small resolution of the Segre cubic and a small resolution of the Coble fourfold,
  the double cover of $\mathbb{P}^4$ ramified along the Castelnuovo--Richmond quartic.

  Homological projective duality then provides a description of the derived categories of linear sections,
  which we will describe to illustrate the theory.
  The case of the Segre cubic and Coble fourfold is non-trivial enough to exhibit interesting behavior,
  whilst being easy enough to explain the general machinery in this special and very classical case.
\end{abstract}

\tableofcontents

\section{Introduction}
Projective duality has been a cornerstone of (algebraic) geometry since a long time,
going back to the 19th century and before.
It provides a way to recover a projective variety from its set of tangent hyperplanes in the dual projective space.
For a modern treatment one is referred to \cite{MR2113135,MR1264417,MR1234494}.

An interesting example is provided by
\begin{enumerate}
  \item[1.] the \emph{Segre cubic} $\segre$,
\end{enumerate}
the up to projective equivalence unique singular cubic threefold with the maximal number of 10~nodes,
studied by Segre in 1887 \cite{segre-cubic}.
An explicit description is recalled in \eqref{equation:segre-cubic-P5}.
Its projective dual is
\begin{enumerate}
  \item[2.] the \emph{Castelnuovo--Richmond quartic} $\igusa$,
\end{enumerate}
an explicit singular quartic threefold,
whose equations are given in \eqref{equation:igusa-quartic-in-P5}.
It was studied by Castelnuovo in 1891 \cite{castelnuovo-quartic}
and independently by Richmond in 1902 \cite{richmond-quartic}.
This threefold is also called \textit{Igusa quartic}\footnote{This might be the more common name nowadays,
  but Dolgachev writes the following in \cite[\S10.3]{MR2964027}:
  ``The quartic hypersurface isomorphic to $\igusa$ is often referred to in modern literature as an \emph{Igusa quartic}
  (apparently, reference \cite{MR1007155} is responsible for this unfortunate terminology).''
  Therefore we will write Castelnuovo--Richmond quartic.},
because its modular properties were studied by Igusa in 1962 \cite{MR0141643},
The third main character is
\begin{enumerate}
  \item[3.] the \emph{Coble fourfold} $\coble$,
\end{enumerate}
the double cover $\pi\colon\coble\to\mathbb{P}^4$ ramified along $\igusa$,
thoroughly studied by Coble in his 1929 book \cite{coble-book}.
We will discuss some of the many interesting geometric and modular properties of $\segre$, $\igusa$ and $\coble$ in \cref{section:segre-igusa-coble}.

Recently Kuznetsov has introduced a homological version of projective duality \cite{MR2354207}.
The original motivation is to study derived categories of linear sections,
and it has blossomed into a rich theory of ``homological projective geometry'' \cite{MR3948688,1902.09824v2,MR4280866,MR3948688}
where classical constructions in projective geometry (such as cones and joins)
have an analogous construction on the level of derived categories.
For an introduction to the theory one is referred to \cite{MR3728631,MR3821163,2111.00527v1}.

The theory roughly states that, for a morphism
$f\colon X\to\mathbb{P}(V)$
where $X$ is a smooth projective variety 
the derived category of a linear section $X\times_{\mathbb{P}(V)}\mathbb{P}(L)$ for some $L\subseteq V$
is described in terms of a \emph{homological projective dual} variety $f^\natural\colon X^\natural\to\mathbb{P}(V^\vee)$ (possibly noncommutative)
and its dual linear sections.
To make this possible
we consider a semiorthogonal decomposition of~$\derived^\bounded(X)$
which is compatible with~$\mathcal{O}_X(1)=f^*\mathcal{O}_{\mathbb{P}(V)}(1)$ in a suitable way,
a so-called \emph{Lefschetz decomposition}.
The general theory is developed in the geometric setting in \cite{MR2354207},
and in complete generality without geometricity conditions in \cite{MR3948688}.
We will briefly introduce the framework in \cref{subsection:lefschetz-and-hpd}.

The main result of the theory is the existence of a homological projective dual
together with a recipe to understand derived categories of hyperplane sections.
The name refers to the fact that,
if $X$ is in fact a closed subvariety of $\mathbb{P}(V)$,
then the classical projective dual arises as the critical locus of the dual morphism $f^\natural$.
However, whilst the theory provides an abstract existence result
it does not give an explicit geometric description of the dual in concrete situations
(much like an explicit description of the classical projective dual is usually hard to come by),
which is necessary to interpret the output of the machinery.

Instances in which the homological projective dual has a nice geometric description are listed in \cite{MR3728631},
and include cases like quadrics,
Grassmannians of lines (the Grassmannian--Pfaffian duality, which is still incomplete),
or Veronese embeddings (the Veronese--Clifford duality).

The goal of this (mostly expository) paper is to give a case study of homological projective duality
for the classically relevant case of the Segre cubic~$\segre$.
It has the benefit that the linear sections appearing are all well-known varieties,
such as (smooth and singular) cubic surfaces,
Kummer quartics,
elliptic curves as plane cubics and double covers.
Because all the objects are commutative
(in particular, no need for noncommutative or categorical resolutions),
reasonably small, completely classical,
and their derived categories are easy to describe
the workings and the output of the abstract machinery becomes tractable.

\paragraph{The main result}
The version of homological projective duality from \cite{MR2354207} works for a morphism~$X\to\mathbb{P}(V)$
where~$X$ is smooth projective.
Because the Segre cubic is singular,
we will recall in \cref{subsection:resolutions} a certain small resolution~$\varpi\colon X\to\segre$
together with a natural map~$f\colon X\to\mathbb{P}(V)$
which is the composition of the resolution and the closed immersion.

For its homological projective dual we will consider
a small resolution of singularities~$\rho\colon Y\to\coble$,
which induces a natural map~$g\colon Y\to\mathbb{P}(V^\vee)$.
The short version of the main result is then:
\begin{theorem}
  \label{theorem:main-theorem-imprecise}
  Let~$f\colon X\to\mathbb{P}(V)$ be the composition of
  the small resolution~$\varpi$
  and the inclusion of the Segre cubic~$\segre$.
  Consider the Lefschetz structure from \eqref{equation:lefschetz-for-segre-cubic}.

  The homological projective dual is~$g\colon Y\to\mathbb{P}(V^\vee)$,
  where~$g$ is the composition of
  the small resolution~$\rho$
  and the double cover ramified along the Castelnuovo--Richmond quartic.
\end{theorem}
In \cref{theorem:main-theorem-precise} we will give the precise version.
The proof of this theorem is given in \cref{subsection:segre-hpd}.
It is an application of homological projective duality for projective bundles \cite[\S8]{MR2354207}
which we will recall in \cref{subsection:hpd-bundles}.

More interesting than the theorem itself
are the applications of the machinery of homological projective duality.
Linear sections of~$X$ and~$Y$ are easy to describe
and often have a very classical interpretation.
The description of their derived categories is performed in \cref{section:linear-sections}.

\paragraph{Comparison with homological projective duality for quadrics}
Arguably the easiest non-trivial case of homological projective duality is that of quadrics \cite{MR4283549}.
On the other hand,
the homological projective dual of a smooth cubic threefold is very complicated,
as explained in \cref{subsection:other-cubics},
and the classical projective dual is highly singular hypersurface of degree~24 (by \eqref{equation:plucker-teissier}).

However, the many singularities of the Segre cubic make its dual hypersurface tractable.
This, together with the existence of a rectangular Lefschetz decomposition for a small resolution,
and an understanding of the geometry in very classical terms,
make it possible to prove \cref{theorem:main-theorem-imprecise},
with a statement analogous to that of odd-dimensional quadrics,
except that it involves singular hypersurfaces of degrees~3 and~4,
and double covers thereof.
We will recall homological projective duality for quadrics in \cref{example:hpd-quadrics},
and the reader is invited to compare this description to the statement of \cref{theorem:main-theorem-precise}.

\paragraph{A second Lefschetz decomposition}
The input for homological projective duality is not just the morphism~$f\colon X\to\mathbb{P}(V)$,
one also needs to specify a Lefschetz structure on~$\derived^\bounded(X)$.
In the context of homogeneous varieties this leads to the search for \emph{minimal} Lefschetz centers \cite{MR3137200,MR4017162},
which are in some sense optimal:
their homological projective duals will be as small as possible.
Many examples of homological projective duality are considered for minimal Lefschetz centers.

The Lefschetz decomposition in \cref{theorem:main-theorem-imprecise} is rectangular,
so the Lefschetz center is in particular minimal, so optimal from the point-of-view of homological projective duality.

But we can construct a \emph{second} rectangular Lefschetz decomposition.
We will first do this by using a different small resolution of~$\segre$ in \cref{proposition:blowup-lefschetz-structure}.
Then in \cref{proposition:quiver-lefschetz-structure} we will construct another rectangular Lefschetz decomposition
using a modular interpretation of~$\segre$ and its resolutions in terms of quiver representations.
We will compare the Lefschetz centers in \cref{proposition:michel},
and leave a description of the homological projective dual for further work.

\paragraph{Conventions}
Throughout~$k$ will be an algebraically closed field of characteristic~0.

\paragraph{Acknowledgements}
We would like to thank Sasha Kuznetsov for many interesting discussions,
and suggesting this particularly nice approach to homological projective duality of the Segre cubic.
This expository writeup can also be seen as a (very late) offshoot of the online learning seminar on homological projective duality,
organised by Daniel Huybrechts in the spring of 2020 at the University of Bonn,
and we would like to thank him for setting this up in very trying times.
The first author is grateful for stimulating discussions about projective geometry with Gebhard Martin,
and we would like to thank Michel Van den Bergh for the nice proof of \cref{proposition:michel}.

The first author is funded by the IMPRS program of the Max--Planck Society.

\section{Homological projective duality}
\label{section:hpd}
We will now give a brief introduction to homological projective duality,
and discuss the case of projective bundles.
For more information the reader is referred to \cite{MR3728631,MR3821163},
or the original \cite{MR2354207}.
In our presentation we will include the noncommutative picture from \cite{MR3948688}.

\subsection{Lefschetz categories and homological projective duality}
\label{subsection:lefschetz-and-hpd}
Let~$X$ be a smooth and proper variety over~$k$,
let~$V$ be a finite-dimensional vector space,
and let~$f\colon X\to\mathbb{P}(V)$ be a morphism.
We set~$\mathcal{O}_X(1)=f^*\mathcal{O}_{\mathbb{P}(V)}(1)$.

A \emph{Lefschetz center} of~$\derived^\bounded(X)$ is an admissible subcategory~$\mathcal{A}_0$ of~$\derived^\bounded(X)$,
for which
\begin{itemize}
  \item there exist admissible subcategories~$\mathcal{A}_1,\ldots,\mathcal{A}_{m-1}$ of~$\derived^\bounded(X)$,
  \item which fit into a chain of admissible subcategories
\end{itemize}
\begin{equation}
  \label{equation:inclusion-chain}
  \begin{gathered}
    0\subseteq\mathcal{A}_{m-1}\subseteq\ldots\subseteq\mathcal{A}_{1}\subseteq\mathcal{A}_0
  \end{gathered}
\end{equation}
such that there exists a semiorthogonal decomposition
\begin{equation}
  \label{equation:lefschetz-sods}
  \derived^\bounded(X)=\langle\mathcal{A}_0,\mathcal{A}_1(1),\ldots,\mathcal{A}_{m-2}(m-2),\mathcal{A}_{m-1}(m-1)\rangle.
\end{equation}
Here we write~$\mathcal{A}_i(i)$ for the image of~$\mathcal{A}_i$ under the tensor product with~$\mathcal{O}_X(i)$.
The~$\mathcal{A}_i$ are the \emph{Lefschetz components},
and the semiorthogonal decomposition in \eqref{equation:lefschetz-sods} is a \emph{Lefschetz decomposition} of~$\derived^\bounded(X)$.
One can moreover show that the choice of~$\mathcal{A}_0$ determines all the components~$\mathcal{A}_i\subseteq\mathcal{A}_0$.

The main motivation for this definition is that,
if one takes a hyperplane~$H\subseteq\mathbb{P}(V)$
such that the ``hyperplane section'' $X\times_{\mathbb{P}(V)}H$
has dimension\footnote{Without this condition one needs to take the derived fiber product. We will discuss derived fiber products in \cref{section:linear-sections}.}
$\dim X-1$,
one gets an induced semiorthogonal decomposition
\begin{equation}
  \label{equation:hyperplane-sod}
  \derived^\bounded(X\times_{\mathbb{P}(V)}H)=\langle\mathcal{C}_H,\mathcal{A}_1(1),\mathcal{A}_2(2),\ldots,\mathcal{A}_{m-2}(m-2),\mathcal{A}_{m-1}(m-1)\rangle
\end{equation}
where one can restrict the subcategories~$\mathcal{A}_i$ to admissible subcategories for the hyperplane section,
and preserve their semiorthogonality as long as one discards~$\mathcal{A}_0$.
The category~$\mathcal{C}_H$ is by definition the right orthogonal to the restricted decomposition.

\paragraph{The homological projective dual}
The main construction of homological projective duality is a family of categories over~$\mathbb{P}(V^\vee)$,
the dual projective space whose points correspond to the hyperplanes~$H$,
whose fibers are the categories~$\mathcal{C}_H$.
To make this construction,
we consider the universal hyperplane section~$\mathbf{H}\subset\mathbb{P}(V)\times\mathbb{P}(V^\vee)$
and set~$\mathbf{H}(X)\colonequals X\times_{\mathbb{P}(V)}\mathbf{H}$.
One can show that there exists a $\mathbb{P}(V^\vee)$-linear\footnote{To be taken in the sense of \cite[\S2.6]{MR2238172}, i.e.~it commutes with tensor products of pullbacks along~$\mathbf{H}(X)\to\mathbb{P}(V^\vee)$ of objects in~$\derived^\bounded(\mathbb{P}(V^\vee)$.}
semiorthogonal decomposition
\begin{equation}
  \derived^\bounded(\mathbf{H}(X))
  =
  \langle
    \derived^\bounded(X)^\natural,
    \mathcal{A}_1(1)\boxtimes\derived^\bounded(\mathbb{P}(V^\vee)),
    \ldots,
    \mathcal{A}_{m-1}(m-1)\boxtimes\derived^\bounded(\mathbb{P}(V^\vee))
  \rangle.
\end{equation}
Here~$\derived^\bounded(X)^\natural$ is defined as the right orthogonal to the ``standard'' components,
and we have surpressed some of the notation for the embedding functors.
The category~$\derived^\bounded(X)^\natural$ is~$\mathbb{P}(V^\vee)$-linear,
and one can construct a canonical Lefschetz structure on it,
whose Lefschetz center~$\mathcal{B}_0$ is equivalent to the Lefschetz center~$\mathcal{A}_0$.

The category~$\derived^\bounded(X)^\natural$ is not necessarily the derived category of a variety.
But if it is we make the following definition.
\begin{definition}
  We say that the \emph{homological projective dual} to~$f\colon X\to\mathbb{P}(V)$ equipped with a Lefschetz decomposition
  is a morphism~$f^\natural\colon X^\natural\to\mathbb{P}(V)$ such that
  \begin{itemize}
    \item $\derived^\bounded(X^\natural)$ is equipped with a Lefschetz decomposition,
    \item there exists a~$\mathbb{P}(V^\vee)$-linear Fourier--Mukai functor from~$\derived^\bounded(X^\natural)$ to
      $\derived^\bounded(\mathbf{H}(X))$,
      i.e.~the kernel needs to live in the derived category of
      the fiber product~$(X\times Y)\times_{\mathbb{P}(V)\times\mathbb{P}(V^\vee)}Q$
      where~$Q$ is the incidence variety,
  \end{itemize}
  which induces a Lefschetz equivalence~$\derived^\bounded(X^\natural)\cong\derived^\bounded(X)^\natural$,
  i.e.~an equivalence which also gives an equivalence of Lefschetz centers,
  between the canonical Lefschetz center from the construction and the Lefschetz center from the first point.
\end{definition}

For the theory to work
it is not actually necessary that~$f^\natural\colon X^\natural\to\mathbb{P}(V)$
is really a morphism from a (smooth proper) variety~$X^\natural$ to the dual projective space.
This will rarely be the case in fact.
Rather it suffices to treat~$\derived^\bounded(X)^\natural$ as an enhanced triangulated category,
which is linear over~$\mathbb{P}(V^\vee)$.
One can then just formally consider this as the homological projective dual.
The general machinery of \cite{MR3948688} makes this possible,
and proves the existence of a homological projective dual without any geometricity conditions.

We can now discuss the two main parts of the main theorem of homological projective duality \cite[Theorem~6.3]{MR2354207}.

\paragraph{First part of the main theorem: Dual Lefschetz decompositions}
Homological projective duality is indeed a duality,
so its output is again a derived category with a Lefschetz structure
to which the machinery can be applied again.

The Lefschetz structure on~$\derived^\bounded(X)^\natural$ is the \emph{dual Lefschetz decomposition}\footnote{We will not go into the details of left and right homological projective duality: by \cite[\S7.4]{MR3948688} these notions agree because we are working with smooth and proper varieties.}
\begin{equation}
  \derived^\bounded(X)^\natural=
  \langle\mathcal{B}_{n-1}(1-n),\ldots,\mathcal{B}_1(-1),\mathcal{B}_0\rangle
\end{equation}
for some Lefschetz center~$\mathcal{B}_0$ which is equivalent to~$\mathcal{A}_0$.
The length~$n$ is given in \cite[equation~(11)]{MR2354207} and reads
\begin{equation}
  n\colonequals\dim V-1-\max\{i\mid\mathcal{A}_i=\mathcal{A}_0\}.
\end{equation}
The relationship between the~$\mathcal{A}_i$ and~$\mathcal{B}_j$ can be described as follows.
The chain of inclusions \eqref{equation:inclusion-chain} allows one to obtain semiorthogonal decompositions
\begin{equation}
  \mathcal{A}_i=\langle\mathfrak{a}_i,\mathcal{A}_{i+1}\rangle=\ldots=\langle\mathfrak{a}_i,\ldots,\mathfrak{a}_{m-1}\rangle
\end{equation}
and the categories~$\mathfrak{a}_i$ are the \emph{primitive categories} of the Lefschetz decomposition.
The categories~$\mathcal{B}_j$ are built using the same primitive categories,
namely there exist semiorthogonal decompositions
\begin{equation}
  \mathcal{B}_j=\langle\mathfrak{a}_0,\ldots,\mathfrak{a}_{\dim V-j-2}\rangle.
\end{equation}

In all examples mentioned in this paper,
except \cref{example:hpd-quadrics},
the Lefschetz decomposition is in fact \emph{rectangular}: $\mathcal{A}_0=\ldots=\mathcal{A}_m$.
This means that there is but one non-zero primitive subcategory,
$\mathfrak{a}_0=\ldots=\mathfrak{a}_{m-2}=0$
and~$\mathfrak{a}_{m-1}=\mathcal{A}_0$.
We have that~$n=\dim V-n$,
and the dual Lefschetz decomposition is also rectangular.

\paragraph{Second part of the main theorem: Linear sections}
The main application of homological projective duality is to provide a description of the derived category of a linear section of~$X$,
in terms of the Lefschetz decomposition \emph{and} the derived category of the dual linear section of~$X^\natural$.
We will from now on write~$Y=X^\natural$.

For this we consider a linear subspace~$L\subseteq V$,
so that we can consider~$L^\perp\subseteq V^\vee$, and define
\begin{equation}
  \label{equation:X-L-Y-L-perp}
  X_L\colonequals X\times_{\mathbb{P}(V)}\mathbb{P}(L),\qquad Y_{L^\perp}\colonequals Y\times_{\mathbb{P}(V^\vee)}\mathbb{P}(L^\perp).
\end{equation}
For now we will assume that~$L\subseteq V$ is \emph{admissible},
i.e.~that~$X_L$ and~$Y_{L^\perp}$ have their expected dimension, so
\begin{equation}
  \dim X_L=\dim X-\dim_k L^\perp,\qquad\dim Y_{L^\perp}=\dim Y-\dim_k L.
\end{equation}
We comment in \cref{remark:derived-fiber-product} on what to do in general.

With this setup,
the second part of the main theorem of homological projective duality states that
the derived categories of~$X_L$ and~$Y_{L^\perp}$ have induced semiorthogonal decompositions
\begin{equation}
  \label{equation:linear-sections-sod}
  \begin{aligned}
    \derived^\bounded(X_L)&=
    \langle
      \mathcal{C}_L,
      \mathcal{A}_{\dim_k L^\perp}(1),
      \ldots,
      \mathcal{A}_{m-1}(m-\dim_k L^\perp)
    \rangle, \\
    \derived^\bounded(Y_{L^\perp})&=
    \langle
      \mathcal{B}_{n-1}(\dim_k L-n),
      \ldots,
      \mathcal{B}_{\dim_k L}(-1),
      \mathcal{C}_L
      \rangle.
  \end{aligned}
\end{equation}
The important observation is that they have the component~$\mathcal{C}_L$ in common.
So understanding a semiorthogonal decomposition for one variety
gives information on the semiorthogonal decomposition of a possibly very different variety,
of ``complementary dimension''.
This of course requires a decent enough understanding of the homological projective dual variety,
and a way to understand~$\mathcal{C}_L$.

\begin{remark}
  \label{remark:derived-fiber-product}
  The admissibility condition that~$X_L$ and~$Y_{L^\perp}$ from \eqref{equation:X-L-Y-L-perp} are of expected dimension
  can be removed using methods from derived algebraic geometry,
  provided that one takes the \emph{derived fiber products}
  \begin{equation}
    \label{equation:derived-fiber-product}
    X_L\colonequals X\times_{\mathbb{P}(V)}^\LLL\mathbb{P}(L),\qquad
    Y_{L^\perp}\colonequals Y\times_{\mathbb{P}(V^\vee)}^\LLL\mathbb{P}(L^\perp).
  \end{equation}
  This is worked out in \cite{MR3948688}.
  From now on~$X_L$ and~$Y_{L^\perp}$ refer to these derived fiber products.

  To explicitly compute the derived fiber product~$X_L$
  we take a locally free resolution of~$\mathcal{O}_{\mathbb{P}(V)}$-modules
  of either~$\mathcal{O}_X$ or~$\mathcal{O}_{\mathbb{P}(L)}$
  which has the structure of a dg algebra of,
  and take the tensor product with of dg~algebras
  instead of the tensor product of algebras
  in the formation of the fiber product.
  In this way we obtain a derived scheme structure on the ordinary fiber product.
  We will discuss how to do this in practice in \cref{section:linear-sections}.
\end{remark}

Let us now illustrate these results in two examples,
which also explain how to build intuition for the structure of the homological projective dual.

\begin{example}
  \label{example:linear-hpd}
  The easiest example of a Lefschetz decomposition is provided by Beilinson's exceptional collection \cite{MR0509388}
  \begin{equation}
    \derived^\bounded(\mathbb{P}(V))=\langle\mathcal{O}_{\mathbb{P}(V)},\mathcal{O}_{\mathbb{P}(V)}(1),\ldots,\mathcal{O}_{\mathbb{P}(V)}(\dim V-1)\rangle
  \end{equation}
  where one takes~$f\colon\mathbb{P}(V)\to\mathbb{P}(V)$ the identity morphism,
  and takes as Lefschetz center
  \begin{equation}
    \label{equation:linear-duality-lefschetz-center}
    \mathcal{A}_0\colonequals\langle\mathcal{O}_{\mathbb{P}(V)}\rangle.
  \end{equation}
  Because a hyperplane section (and linear section) is always just a lower-dimensional projective space,
  for which the restricted decomposition is again a full exceptional collection,
  all the categories~$\mathcal{C}_H$ (and~$\mathcal{C}_L$) are zero,
  because the part in the semiorthogonal decomposition \eqref{equation:hyperplane-sod} (and \eqref{equation:linear-sections-sod})
  determined by the~$\mathcal{A}_i$ already generates the derived category of the section.
  In fact, the homological projective dual of projective space
  (with respect to this choice of Lefschetz center!)
  is just the zero category.

  More interesting to consider is an inclusion of vector spaces~$W\subset V$
  and the resulting morphism~$f\colon\mathbb{P}(W)\hookrightarrow\mathbb{P}(V)$,
  using Beilinson's collection for~$\mathbb{P}(W)$
  and Lefschetz center \eqref{equation:linear-duality-lefschetz-center}.
  Reasoning as in the case of~$W=V$ considered before,
  one can see that the homological projective dual
  is supported on~$\mathbb{P}(W^\perp)\subset\mathbb{P}(V^\vee)$:
  for a point~$h$ outside~$\mathbb{P}(W^\perp)$
  corresponding to a hyperplane~$H\subset\mathbb{P}(V)$
  the intersection~$H\cap\mathbb{P}(W)$ has the expected dimension
  and the restricted collection is again a full exceptional collection.
  One can in fact show that the homological projective dual
  \emph{is} the inclusion~$\mathbb{P}(W^\perp)\subset\mathbb{P}(V^\vee)$ \cite[\S4.1]{MR3728631}.
\end{example}

Whilst especially the first example looks quite pathological
and the generalisation isn't too complicated either,
it turns out that their relative version
(i.e.~for projective bundles)
is exactly the ingredient we need for homological projective duality for the Segre cubic,
and also a convenient way to prove homological projective duality for~$\mathbb{P}(W)\subset\mathbb{P}(V)$.
We will discuss this relative version in detail in \cref{subsection:hpd-bundles}.

\begin{remark}
  We also emphasise that in \cref{example:linear-hpd} we consider a Lefschetz decomposition with respect to~$\mathcal{O}_{\mathbb{P}(V)}(1)$.
  If one instead considers the second Veronese embedding~$\mathbb{P}(V)\hookrightarrow\mathbb{P}(\Sym^2V)$,
  it is possible to consider Lefschetz structures with respect to~$\mathcal{O}_{\mathbb{P}(V)}(2)$.
  This is done in \cite{MR2419925},
  and gives rise to a description of the derived categories of intersections of quadrics,
  as these correspond to linear sections with respect to this choice of line bundle.
\end{remark}

Instead of describing homological projective duality for the Veronese embedding in more detail
we describe homological projective duality for quadrics,
as there are some interesting parallels between this situation and that of the Segre cubic.

\begin{example}
  \label{example:hpd-quadrics}
  Let~$Q\subseteq\mathbb{P}(V)$ be a smooth quadric hypersurface.
  On an odd-dimensional quadric~$Q$ there exists a unique \emph{spinor bundle}~$\mathcal{S}$,
  on an even-dimensional quadric~$Q$ there exist two \emph{spinor bundles}~$\mathcal{S}_+,\mathcal{S}_-$
  which are related by the isomorphism~$\mathcal{S}_{\pm}\otimes\mathcal{O}_Q(1)\cong\mathcal{S}_{\pm}^\vee$ or~$\mathcal{S}_{\mp}^\vee$ depending on the parity of~$\dim Q/2$.
  By mutating Kapranov's decomposition from \cite{MR0939472}
  these bundles give rise to the Lefschetz decompositions
  \begin{equation}
    \label{equation:lefschetz-quadric}
    \derived^\bounded(Q)=
    \begin{cases}
      \langle\mathcal{S},\mathcal{O}_Q,\mathcal{O}_Q(1),\ldots,\mathcal{O}_Q(\dim Q-1)\rangle & \text{$\dim Q$ odd}, \\
      \langle\mathcal{S}_+,\mathcal{O}_Q,\mathcal{S}_+(1),\mathcal{O}_Q(1),\ldots,\mathcal{O}_Q(\dim Q-1)\rangle & \text{$\dim Q$ even}.
    \end{cases}
  \end{equation}
  In both cases the Lefschetz center~$\mathcal{A}_0$ is generated by two objects,
  namely~$\mathcal{A}_0=\langle\mathcal{S},\mathcal{O}_Q\rangle$ resp.~$\langle\mathcal{S}_+,\mathcal{O}_Q\rangle$.
  In the odd-dimensional case we have that~$\mathcal{A}_1=\ldots=\mathcal{A}_{\dim Q-1}=\langle\mathcal{O}_Q\rangle$,
  whilst in the even-dimensional case we have that~$\mathcal{A}_0=\mathcal{A}_1$,
  and~$\mathcal{A}_2=\ldots=\mathcal{A}_{\dim Q-1}=\langle\mathcal{O}_Q\rangle$.

  To heuristically understand the homological projective dual,
  observe that a general hyperplane section~$Q\cap H$ is again a smooth quadric hypersurface.
  Here~$H=\mathbb{P}(L)$ for subspace~$L\subseteq V$ of codimension~1,
  so that~$\mathbb{P}(L^\perp)=\mathrm{pt}$.
  The answer depends on the parity of~$\dim Q\cap H$.

  If the hyperplane section is odd-dimensional,
  then the restriction from the Lefschetz collection on~$Q$
  provides the full exceptional collection in \eqref{equation:lefschetz-quadric}
  because the restriction of a spinor bundle is again a spinor bundle.
  This suggests that the homological projective dual is supported on the classical projective dual,
  because only for singular hyperplane sections a non-trivial contribution can exist.

  If the hyperplane section is even-dimensional,
  then the restriction from the Lefschetz collection on~$Q$
  provides all the twists of~$\mathcal{O}_{Q\cap H}$ in \eqref{equation:lefschetz-quadric},
  but not the two spinor bundles.
  These are orthogonal exceptional objects,
  so that we may interpret~$\mathcal{C}_H$ in \eqref{equation:hyperplane-sod}
  as corresponding to a double cover with discriminant the classical projective dual.

  This heuristic picture can be made fully rigorous,
  and this is indeed the main result of~\cite{MR4283549}.
  It is quite rare that the homological projective dual is supported on the classical projective dual,
  as this requires the full exceptional collection
  to restrict to a full exceptional collection for a smooth hyperplane section.
  Whilst it will turn out not to be the case for the Segre cubic,
  it is conjecturally true for the Cartan cubic in $\mathbb{P}^{26}$,
  see \cite[Conjecture~1.2]{cayley-plane}.
\end{example}

\subsection{Homological projective duality for projective bundles}
\label{subsection:hpd-bundles}
We will now introduce the version of homological projective duality we will need to study the Segre cubic.
To do so we will first recall the setup of \cite[\S8]{MR2354207},
which sets up a relative homological projective duality,
and then we will explain how this leads to the absolute version.

Let~$S$ be a smooth variety,
and~$\mathcal{E}$ a vector bundle of rank~$r$ on~$S$,
such that the dual~$\mathcal{E}^\vee$ is globally generated.
Let~$X=\mathbb{P}_S(\mathcal{E})$ be the projective bundle over~$S$
parameterising line subbundles of~$\mathcal{E}$,
with structure morphism
\begin{equation}
  p\colon X\to S,
\end{equation}
and let~$\mathcal{O}_p(-1)$ denote the tautological line subbundle of~$p^*\mathcal{E}$
such that~$p_*\mathcal{O}_p(1)=\mathcal{E}^\vee$.

We will also need a map to some projective space.
For this we take
\begin{equation}
  V\colonequals\HH^0(X,\mathcal{O}_p(1))^\vee=\HH^0(S,\mathcal{E}^\vee)^\vee,
\end{equation}
so that we can consider the morphism
\begin{equation}
  \label{equation:hpd-bundles-morphism}
  f\colon X\to\mathbb{P}(V).
\end{equation}

Orlov's projective bundle formula gives a $\derived^\bounded(S)$-linear Lefschetz decomposition
\begin{equation}
  \label{equation:lefschetz-decomposition-bundle}
  \derived^\bounded(X)=
  \langle
    \mathcal{A}_0,
    \mathcal{A}_1(1),
    \ldots,
    \mathcal{A}_{\rk\mathcal{E}-1}(\rk\mathcal{E}-1)
  \rangle
\end{equation}
with respect to~$\mathcal{O}_X(1)\colonequals\mathcal{O}_p(1)$,
such that the Lefschetz center is~$\mathcal{A}_0=p^*(\derived^\bounded(S))$.
We have
\begin{equation}
  \mathcal{A}_0=\ldots=\mathcal{A}_{\rk\mathcal{E}-1}=p^*(\derived^\bounded(S)),
\end{equation}
so this is a \emph{rectangular} Lefschetz decomposition.
Observe that~$\mathcal{O}_p(1)\cong f^*\mathcal{O}_{\mathbb{P}(V)}(1)$,
so that we also have a Lefschetz decomposition with respect to \eqref{equation:hpd-bundles-morphism}.

The (relative) homological projective dual to~$X=\mathbb{P}_S(\mathcal{E})$ is constructed using the orthogonal bundle,
i.e.~we set
\begin{equation}
  \label{equation:E-perp}
  \mathcal{E}^\perp\colonequals\ker(V^\vee\otimes\mathcal{O}_S\to\mathcal{E}^\vee)
\end{equation}
and consider~$Y\colonequals\mathbb{P}_S(\mathcal{E}^\perp)$,
with
\begin{equation}
  q\colon Y\to S.
\end{equation}
Equip the derived category with the (again rectangular) $\derived^\bounded(S)$-linear Lefschetz structure
\begin{equation}
  \label{equation:lefschetz-decomposition-orthogonal-bundle}
  \derived^\bounded(Y)=
  \langle
    \mathcal{B}_{\rk\mathcal{E}^\perp-1}(-\dim V+\rk\mathcal{E}+1),
    \ldots,
    \mathcal{B}_1(-1),
    \mathcal{B}_0
  \rangle
\end{equation}
with respect to~$\mathcal{O}_Y(1)\colonequals\mathcal{O}_q(1)$,
where~$\mathcal{B}_0=q^*(\derived^\bounded(S))$.

Again we will need a map to some projective space.
In the definition of~$Y$ we have $q_*\mathcal{O}_q(1)=\mathcal{E}^{\perp,\vee}$
where
\begin{equation}
  0\to\mathcal{E}\to V\otimes\mathcal{O}_S\to\mathcal{E}^{\perp,\vee}\to 0
\end{equation}
is the dual of the defining sequence \eqref{equation:E-perp}
so that~$\HH^0(Y,\mathcal{O}_q(1))=\HH^0(S,\mathcal{E}^{\perp,\vee})=V$,
and we get the morphism
\begin{equation}
  \label{equation:hpd-bundles-dual-morphism}
  g\colon Y\to\mathbb{P}(V^\vee).
\end{equation}
As before, we have that~$g^*\mathcal{O}_{\mathbb{P}(V^\vee)}(1)\cong\mathcal{O}_q(1)$
so \eqref{equation:lefschetz-decomposition-orthogonal-bundle} is also a Lefschetz decomposition
with respect to \eqref{equation:hpd-bundles-dual-morphism}
In \cite[Corollary~8.3]{MR2354207} homological projective duality for projective bundles
is then stated\footnote{There is a typo in the statement: in the notation of op.~cit.~one has to require that $E^*$ is generated by global sections.}
as follows.
\begin{theorem}[Relative version]
  \label{theorem:relative-version}
  Let~$S$ be a smooth variety,
  and~$\mathcal{E}$ a vector bundle of rank~$r$ on~$S$
  such that~$\mathcal{E}^\vee$ is globally generated.
  The \emph{relative} homological projective dual of~$p\colon X\to S$
  (where~$X\colonequals\mathbb{P}_S(\mathcal{E})$)
  with respect to the Lefschetz structure from \eqref{equation:lefschetz-decomposition-bundle},
  is~$q\colon Y\to S$
  (where~$Y\colonequals\mathbb{P}_S(\mathcal{E}^\perp)$)
  with respect to the Lefschetz structure from \eqref{equation:lefschetz-decomposition-orthogonal-bundle}.
\end{theorem}
This relative version of homological projective duality is explained in \cite[Remark~6.28]{MR2354207}:
instead of the morphism~$X\to\mathbb{P}(V)$
the morphism~$X\to\mathbb{P}(V)\times S$ is used.
One can apply this relative version of homological projective duality
using arbitrary morphisms of vector bundles~$\phi\colon\mathcal{F}\to\mathcal{E}^\vee$,
as explained in \cite[Theorem~8.8]{MR2354207}.
This e.g.~leads to a derived equivalence for certain flops,
see Corollary~8.9~in op.~cit.

Rather we want to consider homological projective duality for the morphism \eqref{equation:hpd-bundles-morphism},
so that we will only consider linear sections given by~$L\subseteq V$.

\begin{corollary}[Absolute version]
  \label{corollary:absolute-version}
  With the setup of \cref{theorem:relative-version},
  define~$V\colonequals\HH^0(S,\mathcal{E}^\vee)^\vee$.
  The homological projective dual of~$X\to\mathbb{P}(V)$
  with respect to the Lefschetz structure from \eqref{equation:lefschetz-decomposition-bundle}
  is~$Y\to\mathbb{P}(V^\vee)$,
  with respect to the Lefschetz structure from \eqref{equation:lefschetz-decomposition-orthogonal-bundle}.
\end{corollary}

In the setting of \cref{theorem:relative-version,corollary:absolute-version} the homological projective dual is again a very explicit variety (which a priori does not need to be true at all).
In order to apply this absolute version of homological projective duality,
one needs to determine~$\mathcal{E}^\perp$ and understand the induced morphism~$Y\to\mathbb{P}(V^\vee)$.

The interested reader is referred to \cite{MR3490767} for other applications of \cref{corollary:absolute-version},
where this setup is the main ingredient in the proof of homological projective duality for determinantal varieties.
The geometric input in op.~cit.~is provided by the Springer resolution of the space of matrices of bounded rank,
so that~$\mathcal{E}^\perp$ has an interpretation in terms of representation theory.
In the next section we will describe the geometric input for the Segre cubic.

\section{The Segre cubic, the Castelnuovo--Richmond quartic and the Coble fourfold}
\label{section:segre-igusa-coble}

\paragraph{The Segre cubic $\segre$}
The \emph{Segre cubic}~$\segre$ lives in~$\mathbb{P}^4$,
but we can define it more symmetrically in~$\mathbb{P}^5$ as the subvariety defined by
\begin{equation}
  \label{equation:segre-cubic-P5}
  \left\{
  \begin{aligned}
    0&=\sum_{i=0}^5x_i, \\
    0&=\sum_{i=0}^5x_i^3.
  \end{aligned}
  \right.
\end{equation}
We will consider~$\segre$ as embedded in the invariant hyperplane given by the first equation.
This is a more invariant description than a description directly within~$\mathbb{P}^4$,
as this exhibits~$\segre$ as the zero locus of the first and third symmetric polynomial,
and it is the starting point to prove that~$\Aut(\segre)\cong\sym_6$.

The Segre cubic is characterised as the unique cubic threefold with ten nodes, up to projective equivalence.
For more about its projective geometry we refer to \cref{subsection:projective}.

It appears often in algebraic geometry.
For instance, by Coble \cite{MR1501008} (see also \cite[Theorem~9.4.8]{MR2964027})
$\segre$ is isomorphic to the~GIT~quotient~$(\mathbb{P}^1)^6\sslash\mathrm{SL}_2$.
If one blows up the~10~nodes on~$\segre$
the resulting variety is~$\overline{\mathrm{M}}_{0,6}$,
the Deligne--Mumford compactification of
the moduli space of stable rational curves with~6~marked points \cite[Remark~2.33]{MR4076812}.

It is also the Igusa compactification of
the moduli space of principally polarised abelian surfaces with a level-2 structure,
as explained in \cite[Theorem~IV.1.4]{MR1929793}.
One can moreover show that~$\segre$ is the Satake compactification of~$Y^*(\sqrt{-3})$ \cite[Theorem~1]{MR1274130},
the arithmetic quotient~$(\Gamma(\sqrt{-3})\setminus\mathbb{B}_3)^*$,
where~$\Gamma(\sqrt{-3})$ is the corresponding lattice in~$\mathrm{U}(3,1,\mathcal{O}_{\mathbb{Q}(\sqrt{-3})})$
and~$\mathbb{B}_3$ is the complex~3-ball.

Another modular interpretation of $\segre$ and some of its resolutions,
closely related to the GIT description,
will be discussed in \cref{section:segre-quiver}.

What is relevant to us about the Segre cubic is that its projective dual is of degree~4,
which is remarkably low.
Indeed, the \emph{Pl\"ucker--Teissier formula} \cite[Theorem~1.2.5]{MR2964027} for a hypersurface~$X$ of degree~$d$ in~$\mathbb{P}^n$
with~$m$ nodal singularities (and no other singularities) gives
\begin{equation}
  \label{equation:plucker-teissier}
  \deg X^\vee=d(d-1)^{n-1}-2m
\end{equation}
for the projectively dual hypersurface~$X^\vee$.
We will describe the projective dual of the Segre cubic explicitly next.
On the other hand,
the projective dual of a smooth cubic threefold is a highly singular hypersurface of degree~24,
and we explain in \cref{subsection:other-cubics} why homological projective duality in this case is not very enlightening.

\paragraph{The Castelnuovo--Richmond quartic $\igusa$}
The \emph{Castelnuovo--Richmond quartic} $\igusa$
(or \emph{Igusa quartic})
in $\mathbb{P}^5$ can be defined as the complete intersection by the two equations
\begin{equation}
  \label{equation:igusa-quartic-in-P5}
  \left\{
  \begin{aligned}
    0&=\sum_{i=0}^5x_i, \\
    0&=\left( \sum_{i=0}^5x_i^2 \right)^2-4\sum_{i=0}^5x_i^4.
  \end{aligned}
  \right.
\end{equation}
We will consider~$\igusa$ as embedded in the invariant hyperplane given by the first equation.
By \cite[Theorem~9.4.12]{MR2964027}
$\igusa$ is projectively dual to $\segre$.
It is singular along~15~lines,
which intersect in~15~points,
with~3~lines through each point,
and~3~points on each line,
giving rise to the \emph{Cremona--Richmond configuration}.
We will further discuss its singularities and projective geometry in \cref{subsection:projective}.

This quartic threefold is part of the pencil of~$\sym_6$-invariant quartics in the invariant hyperplane
given by
\begin{equation}
  \left( \sum_{i=0}^5x_i^2 \right)^2-\frac{1}{t}\sum_{i=0}^5x_i^4,
\end{equation}
as explained in \cite[\S1]{MR4076812}.
For~$t=1/4$ we obtain~$\igusa$,
whilst~$t=1/2$ is the \emph{Burkhardt quartic}.
Other values of~$t$ exhibit their own interesting behavior.

As for~$\segre$,
the Castelnuovo--Richmond quartic has a rich modular theory.
From the projective duality
$\segre$ and $\igusa$ are birational.
This manifests itself for instance as
an interpretation of~$\igusa$ as the compactification of the moduli of 6~ordered points on a conic,
and hence also has a quotient description as (another) GIT quotient~$(\mathbb{P}^1)^6\sslash\mathrm{SL}_2$.
It is also the Satake compactification
of the moduli space of principally polarised abelian surfaces with a level-2 structure,
and it is the Satake compactification~$X^*(2)$ \cite[Theorem~1']{MR1274130},
the arithmetic quotient~$(\Gamma(2)\setminus\mathbb{S}_2)^*$,
where~$\Gamma(2)$ is the corresponding lattice in~$\mathrm{Sp}(4,\mathbb{Z})$,
and~$\mathbb{S}_2$ is the Siegel space of degree~2.

\paragraph{The Coble fourfold $\coble$}
The last remaining player to be introduced in this section is the Coble fourfold $\coble$,
defined as the double cover
\begin{equation}
  \pi\colon\coble\to\mathbb{P}^4
\end{equation}
branched along the Castelnuovo--Richmond quartic.

The Coble fourfold also appears often in algebraic geometry.
For instance,
by Coble \cite{MR1501008} (see \cite{MR1007155} for a modern account)
$\coble$ is isomorphic to the~GIT~quotient~$(\mathbb{P}^2)^6\sslash\mathrm{SL}_3$,
and thus related to the moduli space of marked cubic surface
(with the marking corresponding to the ordering of the~6~points being blown up).
It also has a modular interpretation as the Baily--Borel compactification of an arithmetic quotient,
parametrising~K3 surfaces which are double covers of~$\mathbb{P}^2$
branched along an (ordered) set of 6~lines \cite[Theorem~3.5.6]{MR1438547}.

The geometry of its singularities is the same as that of the Castelnuovo--Richmond quartic.

\subsection{Resolutions of singularities}
\label{subsection:resolutions}
In order to apply the theory of homological projective duality as described in \cref{section:hpd}
we need to suitably replace the Segre cubic (and the Coble fourfold),
because the input (at least in the version from \cite{MR2354207}) is required to be smooth and proper.
We comment on a variation without resolution in \cref{subsection:other-cubics}.


The small resolutions of the Segre cubic were classified by Finkelnberg in \cite{MR0914085}.
There are two resolutions in the overview table in \S5 of op.~cit.\ whose automorphism group is~$\sym_5$,
corresponding to type IV and VI.
The geometry of one of these resolutions
is linked to a resolution of the Coble fourfold~$\coble$,
and we will now recall this from \cite[\S2.2]{MR4076812}.
An alternative discussion of the same geometry is given in \cite[\S5]{1906.12295v2}.
We comment on the other resolution in \cref{remark:small-resolutions-segre}.

\paragraph{The quintic del Pezzo surface}
The key in the description of the two resolutions is the del Pezzo surface~$S$ of degree~5.
There is (up to isomorphism) a unique such surface,
isomorphic to the blowup of~$\mathbb{P}^2$ in 4~points,
no three of which are collinear.
Its automorphism group is~$\sym_5$,
whose irreducible representations are described in \cref{table:character-table-s5}.
The most important representation for us will be the irreducible 5-dimensional representation~$W$
with trivial determinant.
It arises,
together with~$W\otimes-1$,
as the summands of~$\Ind_{\sym_4}^{\sym_5}U$ where~$U$ is the unique 2-dimensional irreducible representation of~$\sym_4$.

As recalled in \cite[\S2.2]{MR4076812}, the canonical embedding
\begin{equation}
  S\hookrightarrow\mathbb{P}(\HH^0(S,\omega_S^\vee))=\mathbb{P}(\textstyle\bigwedge^2V)=\mathbb{P}^5
\end{equation}
involves the unique irreducible~$\sym_5$-representation~$\bigwedge^2V$ of dimension~6.
From this description of~$S$ as the (non-complete) intersection of~5~quadrics,
we obtain the irreducible 5-dimensional representation
\begin{equation}
  \HH^0(\mathbb{P}^5,\mathcal{I}_S(2))^\vee=W
\end{equation}
as per \cite[Remark~2.31]{MR4076812}.

Two manifestations of the Grassmannians will play a role in what follows:
$\Gr(2,W^\vee)$ and~$\Gr(3,W)$.
These come equipped with universal subbundles of ranks~2 respectively~3,
which will be denoted~$\mathcal{U}_2$ resp.~$\mathcal{U}_3$.
We have the universal short exact sequences
\begin{equation}
  \label{equation:universal-Gr(2,5)}
  0\to\mathcal{U}_2\to W^\vee\otimes\mathcal{O}_{\Gr(2,W^\vee)}\to\mathcal{Q}_3\to 0
\end{equation}
and
\begin{equation}
  \label{equation:universal-Gr(3,5)}
  0\to\mathcal{U}_3\to W\otimes\mathcal{O}_{\Gr(3,W)}\to\mathcal{Q}_2\to 0.
\end{equation}

Using the embedding $\mathbb{P}(\bigwedge^2V)\hookrightarrow\mathbb{P}(\bigwedge^3W)$
from \cite[Lemma~2.29]{MR4076812} we can describe~$S$ as
the intersection of~$\Gr(3,W)$ in its Pl\"ucker embedding with~$\mathbb{P}(\bigwedge^2V)$.
We remark that the linear $\sym_5$-invariant embedding~$\mathbb{P}(\bigwedge^2V)\hookrightarrow\mathbb{P}(\bigwedge^3W)$
corresponds to the irreducible representation~$V \wedge V$ appearing as the only six-dimensional irreducible representation in~$\bigwedge^3W$,
see also \cite[\S2.4]{MR3851123}.

This allows us to restrict~$\mathcal{U}_2$ and~$\mathcal{U}_3$ to~$S\subseteq\Gr(2,W^\vee)\cong\Gr(3,W)$,
and we will use the same notation for these rank~2 and~3 bundles on~$S$.

\begin{table}
  \centering
  \begin{tabular}{ccccc}
    \toprule
    description                         & notation       & dimension & partition              & determinant \\\midrule
    trivial representation              & $1$            & $1$       & $\ydiagram{5}$         & $1$ \\
    sign representation                 & $-1$           & $1$       & $\ydiagram{1,1,1,1,1}$ & $-1$ \\\addlinespace
    standard representation             & $V$            & $4$       & $\ydiagram{4,1}$       & $-1$ \\\addlinespace
    twisted standard representation     & $V\otimes -1$  & $4$       & $\ydiagram{2,1,1,1}$   & $-1$ \\\addlinespace
    half of induced from $\sym_4$       & $W$            & $5$       & $\ydiagram{3,2}$       & $1$ \\\addlinespace
    other half of induced from $\sym_4$ & $W\otimes -1$  & $5$       & $\ydiagram{2,2,1}$     & $-1$ \\\addlinespace
    exterior square of standard         & $\bigwedge^2V$ & $6$       & $\ydiagram{3,1,1}$     & $-1$ \\
    \bottomrule
  \end{tabular}
  \caption{Irreducible representations of $\sym_5$}
  \label{table:character-table-s5}
\end{table}

\paragraph{Small resolutions of the Segre cubic and Coble fourfold}
We now define the two main varieties,
whose homological projective duality we will prove in \cref{theorem:main-theorem-precise}.
Namely we set
\begin{equation}
  X\colonequals\mathbb{P}_S(\mathcal{U}_2),\qquad Y\colonequals\mathbb{P}_S(\mathcal{U}_3).
\end{equation}
Because~$\mathcal{U}_2$ respectively~$\mathcal{U}_3$ are naturally subbundles of~$W^\vee\otimes\mathcal{O}_S$ respectively~$W\otimes\mathcal{O}_S$
we obtain two ($\sym_5$-equivariant) morphisms:
\begin{equation}
  f\colon X\to\mathbb{P}(W^\vee),\qquad g\colon Y\to\mathbb{P}(W).
\end{equation}
The following proposition explains our interest in these projective bundles.
It is a combination of \cite[Lemma~2.32 and Proposition~2.44]{MR4076812}.
\begin{proposition}[Cheltsov--Kuznetsov--Shramov]
  \label{proposition:resolution}
  The natural map~$f\colon X\to\mathbb{P}(W^\vee)$
  has the Segre cubic~$\segre$ as its image,
  and exhibits~$X$ as a small resolution~$\varpi\colon X\to\segre$.

  The Stein factorisation of the natural map~$g\colon Y\to\mathbb{P}(W)$
  has the Coble fourfold~$\coble$ as intermediate variety,
  and it is the composition of a small resolution~$\rho\colon Y\to\coble$
  and the Coble fourfold's defining double cover ramified in the Castelnuovo--Richmond quartic~$\igusa$.
\end{proposition}

We can summarise the situation in the following diagrams:
for the Segre cubic~$\segre$ we have
\begin{equation}
  \begin{tikzcd}
    & X=\mathbb{P}_S(\mathcal{U}_2) \arrow[ld, "p", "\mathbb{P}^1\text{-bundle}"'] \arrow[rr, "f"] \arrow[rd, "\varpi", "\substack{\text{small} \\ \text{resolution}}"'] & & \mathbb{P}(W^\vee) \\
    S & & \segre \arrow[ru, hook, "i"]
  \end{tikzcd}
\end{equation}
and the Coble fourfold~$\coble$ (and Castelnuovo--Richmond quartic~$\igusa$) we have
\begin{equation}
  \begin{tikzcd}
    & & & \igusa \arrow[d, hook] \\
    & Y=\mathbb{P}_S(\mathcal{U}_3) \arrow[ld, "q", "\mathbb{P}^2\text{-bundle}"'] \arrow[rr, "g"] \arrow[rd, "\rho", "\substack{\text{small} \\ \text{resolution}}"'] & & \mathbb{P}(W) \\
    S & & \coble \arrow[ru, "\pi", "2:1"']
  \end{tikzcd}
\end{equation}

We will discuss some properties of these resolutions in \cref{section:linear-sections},
as we need to understand the properties of the fibers
in order to describe how
(derived categories of) linear sections of~$X$ and~$Y$
are related.

\begin{remark}
  \label{remark:small-resolutions-segre}
  All small resolutions of the Segre cubic are all obtained by blowing up the 10~nodes,
  and contracting a factor of every exceptional~$\mathbb{P}^1\times\mathbb{P}^1$.
  This gives~$2^{10}$ resolutions,
  which can be grouped into~13~isomorphism classes \cite{MR0914085}.
  Of these, 6~correspond to smooth projective varieties,
  the other~7~correspond to smooth proper non-projective varieties.
  They are all related via flops,
  hence by \cite[Theorem~3.9]{alg-geom/9506012}
  (or the more general \cite[Theorem~1.1]{MR1893007})
  they are all derived equivalent.

  The preferred resolution used here,
  and taken from \cite{MR4076812},
  has the benefit of giving an easy-to-describe Lefschetz structure
  for which we can immediately apply an existing result to determine the homological projective dual.
\end{remark}

One could also use another resolution.
By \cite[Proposition~4.6]{MR2605172}
the blowup~$X'\colonequals\mathrm{Bl}_5\mathbb{P}^3$
of~5~points in general position
realises the second isomorphism type
of small resolutions of~$\segre$ with automorphism group~$\sym_5$.
This is again a weak Fano threefold of index~2,
whose half-anticanonical divisor~$\mathcal{O}_{X'}(-\frac{1}{2}\mathrm{K}_{X'})$
in this description is~$\mathcal{O}_{X'}(2H-E)$,
where the~$E_i$ are the exceptional divisors,
we set~$E\colonequals\sum_{i=1}^5E_i$,
and~$H$ is the pullback of~$\mathcal{O}_{\mathbb{P}^3}(1)$.
It is basepoint-free,
and defines a morphism to~$\mathbb{P}^4$ with image the Segre cubic,
which realises it as the resolution of~$\segre$.

\begin{proposition}
  \label{proposition:blowup-lefschetz-structure}
  Let~$X'\to\segre$ be the small resolution of the Segre cubic~$\segre\hookrightarrow\mathbb{P}(W^\vee)$
  given by \cite[Proposition~4.6]{MR2605172},
  so that~$X'=\Bl_5\mathbb{P}^3$ of~5~points in general position.
  Let~$\mathcal{O}_{X'}(H)$ be the pullback of~$\mathcal{O}_{\mathbb{P}^3}(1)$,
  and~$E_1,\ldots,E_5$ the exceptional divisors.
  Then
  \begin{equation}
    \label{equation:blowup-lefschetz-structure}
    \begin{aligned}
      \derived^\bounded(X')&=
      \langle
        \mathcal{O}_{X'},
        \mathcal{O}_{X'}(H),
        \mathcal{O}_{E_i}
        \mid i=1,\ldots,5; \\
        &\qquad
        \mathcal{O}_{X'}(2H-E),
        \mathcal{O}_{X'}(3H-E);
        \mathcal{O}_{E_i}(E_i)
        \mid i=1,\ldots,5
        \rangle,
    \end{aligned}
  \end{equation}
  is a rectangular Lefschetz decomposition of~$\derived^\bounded(X')\cong\derived^\bounded(X)$.
\end{proposition}

\begin{proof}
  From Orlov's blowup formula we have the full exceptional collection
  \begin{equation}
    \derived^\bounded(X')=
    \langle
      \mathcal{O}_{X'},
      \mathcal{O}_{X'}(H),
      \mathcal{O}_{X'}(2H),
      \mathcal{O}_{X'}(3H);
      \mathcal{O}_{E_i},
      \mathcal{O}_{E_i}(-E_i)
      \mid i=1,\ldots,5
    \rangle.
  \end{equation}
  Now right mutate~$\langle\mathcal{O}_{X'}(2H),\mathcal{O}_{X'}(3H)\rangle$
  with respect to~$\langle\mathcal{O}_{E_i}\mid i=1,\ldots,5\rangle$
  to obtain \eqref{equation:blowup-lefschetz-structure}.
  Here we have used that~$\mathcal{O}_{E_i}(H)=\mathcal{O}_{E_i}$
  and~$\mathcal{O}_{E_i}(E_j)=\mathcal{O}_E$ if~$i\neq j$.
\end{proof}
We will come back to this Lefschetz structure in \cref{proposition:lefschetz-comparisons}.
Observe that~$X$ and~$X'$ are related by flops,
hence we have by \cite[Theorem~1.1]{MR1893007}
we have that~$\derived^\bounded(X)\cong\derived^\bounded(X')$.

\subsection{Homological projective duality for the Segre cubic}
\label{subsection:segre-hpd}
We can now state and prove the main theorem,
which was stated somewhat imprecisely in \cref{theorem:main-theorem-imprecise}.
Because the small resolutions discussed in \cref{subsection:resolutions}
are both described by projective bundles over the quintic del Pezzo surface~$S$
we only need to check the relation between the defining bundles
to establish homological projective duality.

\begin{theorem}
  \label{theorem:main-theorem-precise}
  Let~$X\to\segre$ be the small resolution of the Segre cubic~$\segre\hookrightarrow\mathbb{P}(W^\vee)$ given in \cref{proposition:resolution},
  so that~$p\colon X=\mathbb{P}_S(\mathcal{U}_2)\to S$ is a~$\mathbb{P}^1$-bundle over the del Pezzo surface~$S$ of degree~5.
  Consider the morphism~$f\colon X\to\mathbb{P}(W^\vee)$,
  and the rectangular Lefschetz decomposition
  \begin{equation}
    \label{equation:lefschetz-for-segre-cubic}
    \derived^\bounded(X)=
    \langle
      p^*(\derived^\bounded(S)),
      p^*(\derived^\bounded(S))\otimes\mathcal{O}_X(1)
    \rangle
  \end{equation}
  where~$\mathcal{O}_X(1)=f^*(\mathcal{O}_{\mathbb{P}(W^\vee)}(1))$.

  Its homological projective dual is~$g\colon Y\to\mathbb{P}(W)$,
  where~$Y\to\coble$ is the small resolution of the Coble fourfold given in \cref{proposition:resolution},
  so that~$q\colon Y=\mathbb{P}_S(\mathcal{U}_3)\to S$ is a~$\mathbb{P}^2$-bundle over~$S$.
  The dual Lefschetz decomposition is given by
  \begin{equation}
    \derived^\bounded(Y)=
    \langle
      q^*(\derived^\bounded(S))\otimes\mathcal{O}_Y(-2),
      q^*(\derived^\bounded(S))\otimes\mathcal{O}_Y(-1),
      q^*(\derived^\bounded(S))\otimes\mathcal{O}_Y
    \rangle
  \end{equation}
  where~$\mathcal{O}_Y(1)=g^*(\mathcal{O}_{\mathbb{P}(W)}(1))$.
\end{theorem}

\begin{proof}
  By the setup in \cref{subsection:hpd-bundles}
  we need to check that~$\mathcal{U}_2$ and~$\mathcal{U}_3$
  are related by taking the orthogonal bundle \eqref{equation:E-perp}.
  But they are restrictions of the universal subbundle
  on the Grassmannians~$\Gr(2,W^\vee)\cong\Gr(3,W)$.
  Before the restriction we have that~$\mathcal{U}_2^\vee$
  on~$\Gr(2,W^\vee)$
  in \eqref{equation:universal-Gr(2,5)}
  is identified with~$\mathcal{Q}_2$
  on~$\Gr(3,W)$
  in \eqref{equation:universal-Gr(3,5)}
  by this isomorphism,
  so that~$\mathcal{U}_2^\perp$
  is identified with~$\mathcal{U}_3$
  by this isomorphism.
  It now suffices to restrict this identification to~$S$,
  and apply \cref{corollary:absolute-version}.
\end{proof}

\paragraph{Classical projective duality}
The motivation for the terminology \emph{homological projective duality} is explained on \cite[page~159]{MR2354207},
at least when the homological projective dual is an honest variety.
Namely if~$X\to\mathbb{P}(V)$ is a closed immersion
(so that it has a projective dual~$X^\vee\subseteq\mathbb{P}(V^\vee)$ in the usual sense),
then we have an equality of subsets of~$\mathbb{P}(V^\vee)$ between
\begin{itemize}
  \item the projective dual $X^\vee$;
  \item the critical locus of the homological projective dual variety~$Y\to\mathbb{P}(V^\vee)$.
\end{itemize}
Here critical locus refers to the complement of the dense open locus where~$Y\to\mathbb{P}(V^\vee)$ is smooth (by generic smoothness).
This allows homological projective duality to be interpreted as a categorification,
or homological version,
of classical projective duality.

More generally,
in \cite[\S7.4]{MR2354207} a definition of~$X^\vee$ is given
when~$X\to\mathbb{P}(V)$ is not necessarily a closed immersion.
This is called the \emph{classical projective dual},
a terminology explained by \cite[Theorem~7.9]{MR2354207}.
This is further generalised to a noncommutative setting in \cite[\S7.5]{MR3948688}.

We can conclude with the following corollary,
which shows that replacing the Segre cubic with a small resolution
in order to apply the theory of homological projective duality
did not change the link with classical projective duality.
Observe that by dimension reasons the critical locus of~$Y\to\mathbb{P}(W)$
is the critical locus of~$\coble\to\mathbb{P}(W)$,
which is~$\igusa$.
\begin{corollary}
  \label{corollary:classical-projective-dual}
  The classical projective dual of~$f\colon X\to\mathbb{P}(W^\vee)$ is the Castelnuovo--Richmond quartic~$\igusa\subseteq\mathbb{P}(W)$.
  The classical projective dual of~$g\colon Y\to\mathbb{P}(W)$ is the Segre cubic~$\segre\subseteq\mathbb{P}(W^\vee)$.
  In particular,
  for all~$L\subseteq W^\vee$ the following are equivalent:
  \begin{itemize}
    \item $X_L$ is singular;
    \item $Y_{L^\perp}$ is singular.
  \end{itemize}
\end{corollary}
Here we take the \emph{derived fiber product} in the sense of derived algebraic geometry,
so that smooth and singular are to be taken in the sense of \cite[\S4.7]{MR3948688}.

\begin{remark}
  We observe that for the equivalence in \cref{corollary:classical-projective-dual}
  it is important that we use resolutions
  and derived fiber products.
  As we will explain in \cref{lemma:smooth-segre-hyperplane-criterion},
  for a point~$H\in\mathbb{P}(W)$
  we have that the corresponding hyperplane section of~$\segre$
  is smooth if and only if it lies outside the union of
  $\igusa$ and the~10~hyperplanes~$P_i$ corresponding to the~10~nodes,
  making the classical projective dual too big.
  But for a point~$H$ of~$P_i\setminus\igusa$
  we have that~$H\cap\coble=Y_{L^\perp}=\{\text{2 points}\}$ is smooth
  (and the fiber product agrees with the derived fiber product).
  By considering the resolution~$X\to\mathbb{P}(W^\vee)$
  instead of~$\segre\subseteq\mathbb{P}(W^\vee)$
  the hyperplane section~$X_H$ is smooth
  as we will explain in \cref{subsection:hyperplane-segre}.

  Secondly,
  for most points in the singular locus of~$\igusa$
  the ordinary fiber product is a smooth rational curve by \cite[Proposition~2.44]{MR4076812}.
  But the relative dimension of the morphism~$Y\to\mathbb{P}(W)$ is~0,
  not~1.
\end{remark}

\begin{remark}
  If~$L$ is admissible in the sense of \cite[Definition~6.2]{MR2354207},
  i.e.~\emph{both} derived fiber products are in fact underived
  because they are of the expected dimension,
  \cite[Theorem~7.12]{MR2354207} gives an identification of the singularity categories
  of~$X\times_{\mathbb{P}(W^\vee)}\mathbb{P}(L)$
  and~$Y\times_{\mathbb{P}(V^\vee)}\mathbb{P}(L^\perp)$.
  If~$L$ is not admissible
  (because at least one fiber product does not have the expected dimension)
  then for the derived fiber product(s) one has to use a notion of singularity category from derived algebraic geometry.
  We will not go into this.
\end{remark}

\subsection{What about other cubic threefolds?}
\label{subsection:other-cubics}
We now discuss some possible variations on the theme of homological projective duality for our chosen resolution of the Segre cubic.

\paragraph{Smooth cubics}
If one would like to study homological projective duality for a \emph{smooth} cubic threefold,
one runs into the following problem.
Let~$X$ denote a smooth cubic threefold for now, then we have the semiorthogonal decomposition
\begin{equation}
  \label{equation:smooth-cubic-threefold-sod}
  \derived^\bounded(X)=\langle\mathcal{R}_X,\mathcal{O}_X,\mathcal{O}_X(1)\rangle
\end{equation}
where~$\mathcal{R}_X$ is a~$\frac{5}{3}$\dash Calabi--Yau category \cite[Corollary~4.1]{MR3987870}.

Whilst it is not clear whether~$\mathcal{R}_X$ is indecomposable,
the decomposition \eqref{equation:smooth-cubic-threefold-sod}
cannot in any way be refined to an interesting Lefschetz decomposition.
By considering the Hodge numbers of~$X$ together with the additivity of Hochschild homology
we see that the longest length of an exceptional collection is~4,
so~$\mathcal{A}_1$ can at most consist of~2~exceptional objects,
hence the interesting component (if such a Lefschetz decomposition exists at all)
consists of~7~exceptional objects with a complicated structure.

If on the other hand we were to take~$\langle\mathcal{R}_X,\mathcal{O}_X\rangle$ as the initial block,
so that~$\mathcal{A}_1=\langle\mathcal{O}_X\rangle$,
then the only contribution to the derived category of a hyperplane section is a single exceptional object.
Hence the homological projective dual is a very complicated object,
and we refer to \cite[\S4]{MR3626552}
for a description of derived categories of hypersurfaces
in terms of gauged Landau--Ginzburg models,
which can be used as a possible starting point for a description
of the interesting component
(which in this case is a~$\frac{4}{3}$\dash Calabi--Yau category consisting of~8 exceptional objects).

One could argue that homological projective duality is a balancing act,
where one tries to make the input data and the resulting homological projective dual
have roughly equal complexity,
so that one can leverage information about one to understand the other.
The resolution of the Segre cubic seems to provide a particularly good solution to this balancing act.

\paragraph{The unresolved Segre cubic}
The smoothness assumption for the initial input which is present in \cite{MR2354207}
has been removed in \cite{MR3948688}.
But currently lacking an interesting Lefschetz decomposition for~$\derived^\bounded(\segre)$
to start the machinery with
we have focused in this article on the smooth case.
It would be interesting to find a similar Lefschetz decomposition for the Segre cubic,
or rule out the existence of one.

\paragraph{Resolutions of singular cubic threefolds with fewer nodes}
If instead of the Segre cubic or a smooth cubic threefold
one wishes to analyse homological projective duality for (a resolution of) a cubic threefold with 1 to 9 nodes,
one can consider the description and classification obtained in \cite{MR1005050}.
A starting point for the description of the derived categories of the singular cubics
is provided by \cite{MR4277855}.


Starting with the 1-nodal case,
by \cite[Proposition~4.6]{MR3050698}
the derived category of a small resolution
has a semiorthogonal decomposition
in terms of the~2 exceptional line bundles~$\mathcal{O}_X,\mathcal{O}_X(1)$
and the derived category of a genus~4 curve.
This brings us in a situation similar to that of the smooth cubic
and the homological projective dual is seemingly a complicated object.

But if the nodal cubic threefold has a \emph{determinantal} presentation,
we can bootstrap from homological projective duality for determinantal varieties \cite{MR3490767}.

Consider first the 6-nodal case,
which is the general determinantal cubic threefold
and corresponds to the case $J_9$ in \cite{MR1005050}.
The derived category of a small resolution has a semiorthogonal decomposition
in terms of the~2 exceptional line bundles~$\mathcal{O}_X,\mathcal{O}_X(1)$
and~4~additional exceptional objects \cite[Remark~6.10]{MR3490767}.
Assume for now that we can turn these~6~exceptional objects into a rectangular Lefschetz decomposition,
then the theory provides~3~exceptional objects for the derived category of a hyperplane section.
A general section is a smooth cubic surface,
for which a full exceptional collection contains~9~exceptional objects.
Therefore the general fiber of the homological projective dual~$X^\natural\to\mathbb{P}(V^\vee)$
consists of~6~exceptional objects.

We can make this prediction precise by the following proposition,
which also covers cases with more nodes.
We recall the setup for homological projective duality for determinantal hypersurfaces from \cite[\S5, \S6.3]{MR3490767}
for this specific instance.
\begin{setup}
  Let~$A$ and~$B$ be~3\dash dimensional vector spaces,
  and set~$V\colonequals A\otimes_kB$.
  Let~$X\subset\mathbb{P}(V^\vee)=\mathbb{P}^8$ be the cubic determinantal~7\dash fold 
  corresponding to the locus of~$3\times 3$-matrices of rank at most~2.
  Let~$\widetilde{X}\to X$ be its Springer resolution,
  given as the projectivisation of~$\Omega_{\mathbb{P}(A)}^1(1)\otimes_kB$.
  We will consider the composition
  \begin{equation}
    f\colon\widetilde{X}\to\mathbb{P}(V^\vee).
  \end{equation}
  On the dual side we let
  \begin{equation}
    g\colon Y=\mathbb{P}^2\times\mathbb{P}^2=\mathbb{P}(A)\times\mathbb{P}(B)\hookrightarrow\mathbb{P}(V) 
  \end{equation}
  be the Segre embedding.
  These have rectangular Lefschetz decompositions
  \begin{equation}
    \begin{aligned}
      \derived^\bounded(\widetilde{X})&=\langle\mathcal{A}_0,\mathcal{A}_1(1),\mathcal{A}_2(2),\mathcal{A}_3(3),\mathcal{A}_4(4),\mathcal{A}_5(5)\rangle \\
      \derived^\bounded(Y)&=\langle\mathcal{B}_2(-2),\mathcal{B}_1(-1),\mathcal{B}_0\rangle
    \end{aligned}
  \end{equation}
  with Lefschetz center~$\mathcal{A}_0=\mathcal{B}_0=\derived^\bounded(\mathbb{P}(A))$,
  and by \cite[Theorem~3.5]{MR3490767} we have homological projective duality
  for~$\widetilde{X}$ and~$Y$ with respect to these choices.
  This result is another instance of homological projective duality for projective bundles
  given in \cref{subsection:hpd-bundles}.
\end{setup}

We need to reduce the case of cubic in~$\mathbb{P}^8$
to that of a cubic in~$\mathbb{P}^4$,
and deduce homological projective duality for this linear section.
\begin{proposition}
  Let~$L\subset V^\vee$ be a~5\dash dimensional linear subspace,
  such that~$Y\cap\mathbb{P}(L^\perp)=\emptyset$,
  and hence~$X\cap\mathbb{P}(L)$ is a 6-nodal determinantal cubic threefold.

  The linear projection~$g_L\colon Y\to\mathbb{P}(L^\vee)$
  is homological projective dual to the restriction~$f_L\colon\widetilde{X}_L\to\mathbb{P}(L)$,
  where~$\derived^\bounded(Y)$ and~$\derived^\bounded(\widetilde{X}_L)$
  are equipped with the rectangular Lefschetz decompositions
  \begin{equation}
    \begin{aligned}
      \derived^\bounded(\widetilde{X}_L)&=\langle\mathcal{A}_0,\mathcal{A}_1(1)\rangle \\
      \derived^\bounded(Y)&=\langle\mathcal{B}_2(-2),\mathcal{B}_1(-1),\mathcal{B}_0\rangle
    \end{aligned}
  \end{equation}
  with~$\mathcal{A}_0=\mathcal{B}_0=\derived^\bounded(\mathbb{P}^2)$.

\end{proposition}

\begin{proof}
  This follows from applying homological projective duality for linear systems with a base locus \cite[Theorem~1.1]{MR4176839}
  (see also \cite[\S A.2]{MR4280866} for an abstract version with empty base locus),
  with the roles of~$\widetilde{X}$ and~$Y$ reversed in the notation of op.~cit.
  The linear section~$\widetilde{X}_L$ is the crepant resolution of a determinantal cubic threefold,
  and~$Y_{L^\perp}=\emptyset$ by our assumption.
  These are smooth,
  and satisfy the expected dimension condition,
  hence we obtain the stated homological projective duality for~$\widetilde{X}_L$
  and~$Y=\Bl_{Y_{L^\perp}}Y$.
\end{proof}

We have that the linear projection~$Y=\mathbb{P}^2\times\mathbb{P}^2\to\mathbb{P}(L)=\mathbb{P}^4$
is a~$6:1$-cover,
ramified along the classical projective dual of the cubic.
By the Pl\"ucker--Teissier formula \eqref{equation:plucker-teissier}
this is a hypersurface of degree~12.
Because this is a~$6:1$\dash cover,
we have that a smooth hyperplane section of~$\widetilde{X}$
has~3~exceptional objects coming from the Lefschetz center~$\mathcal{A}_0=\derived^\bounded(\mathbb{P}^2)$,
and~6~orthogonal exceptional objects from the cover,
which is an incarnation of Orlov's blowup formula
for the smooth cubic surface~$\widetilde{X}_L\cong\Bl_6\mathbb{P}^2$.

We can also consider a general plane section of~$\widetilde{X}_L$,
which is a smooth cubic curve.
On the dual side we have a line section of~$Y_{L^\perp}$,
which is a~$6:1$\dash cover of~$\mathbb{P}^1$
ramified along~12~points,
and by the Riemann--Hurwitz formula this is a curve of genus~1.
Homological projective duality thus provides
an equivalence of categories for these curves by \eqref{equation:linear-sections-sod}.

\section{Applying homological projective duality}
\label{section:linear-sections}
In this section we will describe linear sections of the resolved Segre cubic
and its homological projective dual.
This illustrates how the abstract machinery describes
derived categories of linear sections,
which in this setting have an explicit and classical description.
We will describe two methods for this:
\begin{itemize}
  \item a bottom-up approach,
    starting from a linear section of~$\segre$ resp.~$\coble$,
    and then understanding how the resolution affects the description;
  \item a top-down approach,
    starting from a linear section of~$X$ resp.~$Y$,
    and then understanding how the restriction of the resolution can be interpreted.
\end{itemize}

We will be exhaustive in our discussion of
hyperplane sections of the (resolved) Segre cubic in \cref{subsection:hyperplane-segre},
but restrict ourselves to only discussing
some interesting examples in the other settings in \cref{subsection:other-linear-sections},
as the amount of cases to be covered in a complete case-by-case analysis is large.
For the analysis we need to understand both
the projective geometry of $\segre$, $\igusa$ and $\coble$
(which we recall in \cref{subsection:projective})
and the geometry of the resolutions defined in \cref{subsection:resolutions}.

Before we do this we will explain
the recipe of describing the derived category in the ideal situation,
when the resolutions do not play a role in the description.

\begin{definition}
  We will say that~$L\subseteq W^\vee$ of dimension~$2,3,4$ is \emph{generic} if
  \begin{itemize}
    \item $\mathbb{P}(L)\cap\segre$ is smooth of dimension~$\dim_kL-2$,
      and it avoids the singular locus of~$\segre$;
    \item $\mathbb{P}(L^\perp)\cap\igusa$ is smooth of dimension~$3-\dim_kL$,
      and it avoids the singular locus of~$\igusa$.
  \end{itemize}
\end{definition}
What happens in this case is that~$\mathbb{P}(L)\cap\segre\cong X_L$
and~$\mathbb{P}(L^\perp)\times_{\mathbb{P}(W)}\coble\cong Y_{L^\perp}$.
This makes describing the output of homological projective duality
using the bottom-up approach straightforward.
If~$\dim_kL=1$ then~$\mathbb{P}(L^\perp)\cap\igusa$ always hits
the singular locus of the Castelnuovo--Richmond quartic
as we will explain in \cref{subsection:projective},
so there is no generic~$L\subseteq W^\vee$ in this case.

We have summarised the description in \cref{table:generic}.
Let us explain what is written there.
If~$\dim_kL=2$,
then on the Segre side we see that~$X_L$ consists of~3 points,
whilst on the Coble side we obtain a del Pezzo surface~$Y_{L^\perp}$ of degree~2
as the double cover of~$\mathbb{P}(L^\perp)\cong\mathbb{P}^2$
ramified in the smooth quartic curve~$\mathbb{P}(L^\perp)\cap\igusa$
which is isomorphic to the blowup~$\mathrm{Bl}_7\mathbb{P}^2$
of~7~points in general position.
By \eqref{equation:linear-sections-sod} we get
\begin{equation}
  \begin{aligned}
    \derived^\bounded(X_L)&=\mathcal{C}_L=\langle E_1,E_2,E_3\rangle \\
    \derived^\bounded(Y_{L^\perp})&=\langle\mathcal{B}_1(-1),\mathcal{C}_L\rangle=\langle\derived^\bounded(S), E_1,E_2,E_3\rangle.
  \end{aligned}
\end{equation}
Here~$E_1,E_2,E_3$ are completely orthogonal exceptional objects.
In fact,
the composition~$Y_{L^\perp}\hookrightarrow Y\twoheadrightarrow S$
obtained similarly to \eqref{equation:sasha-composition}
exhibits~$Y_{L^\perp}$ as~$\operatorname{Bl}_3S$,
and thus the semiorthogonal decomposition induced by homological projective duality
can be interpreted as an instance of Orlov's blowup formula.

If~$\dim_kL=3$,
then on the Segre side we obtain a smooth cubic curve~$X_L$,
whilst on the Coble side we obtain a double cover~$Y_{L^\perp}$
of~$\mathbb{P}(L^\perp)\cong\mathbb{P}^1$
ramified in the~4 points~$\mathbb{P}(L^\perp)\cap\igusa$,
which is \emph{also} a curve of genus~1.
By \eqref{equation:linear-sections-sod} we get
\begin{equation}
  \derived^\bounded(X_L)\cong\derived^\bounded(Y_{L^\perp})
\end{equation}
as there are no contributions by the Lefschetz components on either side.
By the reconstruction of curves from their derived category (see e.g.~\cite[Corollary~5.46]{MR2244106})
we even obtain~$X_L\cong Y_{L^\perp}$.

If~$\dim_kL=4$,
then on the Segre side we obtain
a smooth cubic surface,
whilst on the Coble side we see that~$Y_{L^\perp}$ consists of~2 points,
as~$\mathbb{P}(L^\perp)\cap\igusa=\emptyset$.
By \eqref{equation:linear-sections-sod} we get
\begin{equation}
  \begin{aligned}
    \derived^\bounded(X_L)&=\langle\mathcal{C}_L,\mathcal{A}_1(1)\rangle=\langle E_1,E_2,\derived^\bounded(S)\rangle \\
    \derived^\bounded(Y_{L^\perp})&=\mathcal{C}_L=\langle E_1,E_2\rangle.
  \end{aligned}
\end{equation}
Here~$E_1,E_2$ are completely orthogonal exceptional objects.
As for the case of~$\dim_kL=2$
we can consider the composition~$X_L\hookrightarrow X\twoheadrightarrow S$
obtained similarly to \eqref{equation:sasha-composition}
and this exhibits~$X_L$ as~$\operatorname{Bl}_2S$,
so the same comment as for~$Y_{L^\perp}$ with~$\dim_kL=2$ applies.

\begin{table}
  \centering
  \begin{adjustbox}{center}
    \begin{tabular}{cc|ccc}
      \toprule
      $\dim_k L$    & $\mathbb{P}(L)\cap\segre$                    & $\mathbb{P}(L^\perp)\cap\igusa$ & $\mathbb{P}(L^\perp)\times_{\mathbb{P}(W)}\coble$ \\
                    & $\derived^\bounded(\mathbb{P}(L)\cap\segre)$ &                                 & $\derived^\bounded(\mathbb{P}(L^\perp)\times_{\mathbb{P}(W)}\coble)$ \\
      \midrule
      2             & 3 points                                     & smooth quartic curve            & double cover of $\mathbb{P}^2$ ramified in quartic \\
                    & 3 orthogonal objects                         &                                 & $10=7+3$ exceptional objects \\
      \addlinespace
      3             & smooth cubic curve $C$                       & 4 points                        & double cover $C$ of $\mathbb{P}^1$ ramified in 4 points \\
                    & $\derived^\bounded(C)$                       &                                 & $\derived^\bounded(C)$ \\
      \addlinespace
      4             & smooth cubic surface                         & $\emptyset$                     & 2 points \\
                    & $9=7+2$ exceptional objects                  &                                 & 2 orthogonal objects \\
      \bottomrule
    \end{tabular}
  \end{adjustbox}
  \caption{Description of the generic situation}
  \label{table:generic}
\end{table}

In what follows we will discuss what happens in the \emph{non-generic} situation,
and explain how the resolutions~$\varpi\colon X\to\segre$ and~$\rho\colon Y\to\coble$
change the resulting descriptions of (some) linear sections and their derived categories.

\subsection{The projective geometry of \texorpdfstring{$\segre$}{the Segre cubic}, \texorpdfstring{$\igusa$}{the Castelnuovo--Richmond quartic} and \texorpdfstring{$\coble$}{the Coble fourfold}}
\label{subsection:projective}
We will now discuss the geometry of the Segre cubic,
Castelnuovo--Richmond quartic and Coble fourfold in more detail,
to prepare for the description and (partial) classification of
linear sections of~$X$ and~$Y$.

\paragraph{Segre cubic}
The Segre cubic contains ten nodes $p_1, \dots, p_{10}$,
which are the $\sym_6$-orbit of the point $(1:1:1:-1:-1:-1)$.
Moreover,
$\segre$ contains exactly 15 planes,
called \emph{Segre planes}, defined as
\begin{equation}
  P_\sigma\colonequals\{x_{\sigma(0)} + x_{\sigma(3)}=x_{\sigma(1)}+ x_{\sigma(4)}=x_{\sigma(2)} + x_{\sigma(5)}=0\}
\end{equation}
for $\sigma \in \sym_6$.
No three of the ten nodes are collinear,
and each Segre plane contains exactly four of the ten nodes.
Each node is contained in exactly six Segre planes.
This gives a~$(15_4,10_6)$-configuration.

Let us also mention two sets of hyperplanes in~$\mathbb{P}^5$
(which define hyperplanes in~$\mathbb{P}(W^\vee)$).
The first set is given by the~15~hyperplanes
\begin{equation}
  T_{i,j}\colonequals\{x_i - x_j=0\}\qquad 0\leq i<j\leq 5,
\end{equation}
whose intersection with $\segre$ is the Cayley cubic surface,
the unique cubic surface with~4 nodes.
The other set of~15~hyperplanes is given by
\begin{equation}
  H_{i,j}\colonequals\{x_i+x_j=0\}\qquad 0\leq i<j\leq 5,
\end{equation}
which have the property that they each contain exactly three of the fifteen Segre planes
so that their intersection with~$\segre$ is the union of three projective planes.

\paragraph{Castelnuovo--Richmond quartic}
The Castelnuovo--Richmond quartic $\igusa$ is singular along fifteen lines~$\ell_1,\ldots,\ell_{15}$.
The singular set of $\igusa[\sing]$ consists of fifteen points $q_1,\ldots,q_{15}$ dual to the~$H_{i,j}$.
and each~$q_i$ is contained in 3~lines,
so that each~$\ell_i$ intersects~6~other lines.
This makes up the \emph{Cremona--Richmond configuration},
which has type~$(15_3,15_3)$.

The Segre cubic and the Castelnuovo--Richmond quartic
are projectively dual in the classical sense
and we will use the duality throughout the next section.
We will therefore elaborate here a bit about it.

There are ten hyperplanes~$P_i \subseteq \mathbb{P}(W)$
which correspond to hyperplanes containing the nodes~$p_i\in\segre$.
The hyperplanes~$P_i$ intersect the Castelnuovo--Richmond quartic
in a smooth quadric $Q_i$ with non-reduced structure.

The fifteen singular lines of $\igusa$ are exactly the projective duals of the fifteen Segre planes.
Their fifteen intersection points are dual to the fifteen hyperplanes~$H_{i,j}\subseteq\mathbb{P}(W)$.


We can also describe the (rational) duality map
\begin{equation}
  d \colon \segre \dashrightarrow \igusa.
\end{equation}
The map $d$ restricted to the complement of the fifteen Segre planes
is an isomorphism onto the complement of the ten hyperplanes $P_i$, i.e.
\begin{equation}
  \label{equation:projective-duality-isomorphism}
  d\colon \segre \setminus \bigcup_{\sigma\in\sym_6} P_\sigma \overset{\sim}{\to} \igusa \setminus \bigcup_{i=1,\ldots,10}P_i,
\end{equation}
see \cite[Section 3.3.4]{MR1438547}.
There are several implications of this, from which we mention two.

Firstly,
for two points $x\neq y\in \segre$
the tangent hyperplanes $\tangent_x \segre, \tangent_y \segre$
agree if and only if both points are contained in a common $P_\sigma \subset \segre$
and the line joining $x$ and $y$ passes through one of the four nodes contained in $P_\sigma$.
Secondly,
any tangent hyperplane $\tangent_x\segre$
at a point $x\in \segre \setminus \bigcup_{\sigma\in\sym_6} P_\sigma$ does not contain a node.

\paragraph{Coble fourfold}
Recall that the Coble fourfold $\coble$ was defined as
the double cover $\pi \colon \coble \to \mathbb{P}^4$
branched along the Castelnuovo--Richmond quartic~$\igusa$.
It follows from the definition that the singular locus of $\coble$
is isomorphic to the singular locus of $\igusa$.


\subsection{Hyperplane sections of the (resolved) Segre cubic}
\label{subsection:hyperplane-segre}
We can now describe hyperplane sections of~$X$
and their derived categories
using the analysis of the projective geometry of~$\segre$ and~$\igusa$
and their resolutions.

\paragraph{Bottom-up approach}
For this approach we consider hyperplane sections of~$\segre$,
we have to explain which type of singularities a hyperplane section can obtain,
and then how the restriction of the resolution interacts with the hyperplane section.
Depending on the origin of the singularity,
the resolution will interact differently with the hyperplane section.
We have summarised the conclusions in \cref{table:segre-hyperplane}.

We write~$H=\mathbb{P}(L)\subset\mathbb{P}(W^\vee)$,
where~$L\subset W^\vee$ is a codimension-one subspace,
whilst we denote the projective dual by~$h=\mathrm{pt}=\mathbb{P}(L^\perp)\in\mathbb{P}(W)$,
and we will write
\begin{equation}
  X_H\colonequals X\times_{\mathbb{P}(W^\vee)}\mathbb{P}(L),
  \qquad Y_h\colonequals Y\times_{\mathbb{P}(W^\vee)}\mathbb{P}(L^\perp).
\end{equation}
These are denoted~$X_L$ and~$Y_{L^\perp}$ in \cref{section:hpd}.

\begin{lemma}
  \label{lemma:smooth-segre-hyperplane-criterion}
  A hyperplane section $H\cap \segre\subset \mathbb{P}(W^\vee)$ is smooth
  if and only if
  \begin{equation}
    h\in \mathbb{P}(W)\setminus \left( \bigcup_{i=1,\ldots,10} P_i \cup \igusa \right).
  \end{equation}
\end{lemma}

\begin{proof}
  This follows from classical projective duality for varieties with isolated nodal singularities.
  Any hyperplane $h$ which lies in $P_i$ for some $i$
  will inherit the node $p_i$
  and any hyperplane $h\in \igusa$ will be tangent to $\segre$ at some point.
\end{proof}

The restriction of the resolution~$\rho$ in this case is an isomorphism and one obtains a smooth cubic surface.
This is precisely the generic situation already discussed in \cref{subsection:projective}.

The following two lemmas describe the cases
in which there is a single isolated singularity.

\begin{lemma}
  \label{lemma:one-nodal-hyperplane-segre}
  A hyperplane section $H\cap \segre\subset\mathbb{P}(W^\vee)$ is a one-nodal cubic surface
  if and only if
  $h\in P_i\setminus (\igusa \cup\bigcup_{j\neq i} P_j)$ or $h \in \igusa \setminus \bigcup_{i=1,\ldots,10} P_i$.
\end{lemma}

\begin{proof}
  Let us first show that these two cases yield one-nodal cubic surfaces.

  In the first case,
  the point $h$ will lie in some $P_i$
  and therefore the intersection $\segre \cap H$ contains the node $p_i$.
  Since we assume that the hyperplane $H$ contains no other node,
  we also deduce that it cannot be tangent at a point contained in a Segre plane,
  since such a hyperplane would contain the Segre plane and therefore four nodes.
  Moreover,
  the duality map restricts to an isomorphism between $\segre \setminus \bigcup_{\sigma\in\sym_6} P_\sigma$
  and $\igusa \setminus \bigcup_{i=1,\ldots,10} P_i$
  and therefore the hyperplane $h\in P_i$ cannot be tangent at a point not contained in a Segre plane.
  To conclude,
  since the hyperplane $h$ does not lie on~$\igusa$,
  the hyperplane $H$ does not acquire a singularity worse than a node at the point $p_i$.

  This also shows that a tangent hyperplane~$\tangent_p\segre$
  at a point $p \in \segre \setminus \bigcup_{\sigma\in\sym_6} P_\sigma$
  can never contain a node $p_i$ and is only tangent at the point $p$.
  Indeed, the duality map shows that $\tangent_p\segre$ at a point $p\in \segre \setminus \bigcup_\sigma P_\sigma$
  does not contain any of the ten nodes of the Segre cubic.
  Since the duality map is an isomorphism when restricted to the complement of the Segre planes,
  we infer that it will only be tangent at the point $p$.

  Conversely,
  if a hyperplane section is a one-nodal cubic surface,
  then it must intersect the Segre cubic non-transversely or contain a node.

  To finish the proof,
  just note that a hyperplane $H$ intersecting the Segre cubic at a general point $p\in P_\sigma$ not transversely
  will contain the whole Segre plane $P_\sigma \subset H$.
\end{proof}

The two cases in \cref{lemma:one-nodal-hyperplane-segre} behave differently with respect to the resolution~$\rho$.
In the latter case,
the resolution is an isomorphism and we have a 1-nodal cubic surface.
In the first case,
we blow up the node of the cubic surface
and end up with a weak del Pezzo surface.

On the dual side, the picture is also twofold.
In the latter case,
the $h$ is contained in the smooth locus of the Castelnuovo--Richmond quartic and
\begin{equation}
  Y_h = \coble\times_{\mathbb{P}(W)} h = \Spec k[\epsilon]/(\epsilon^2).
\end{equation}
In the first case~$h\in P_i$ does not lie on $\igusa$.
Therefore
\begin{equation}
  Y_h = \{\text{2 points}\}.
\end{equation}

The last possible case for a hyperplane section $H\cap \igusa$ to have exactly one isolated singularity is the following.
\begin{lemma}
  The hyperplane section $H\cap \segre\subset\mathbb{P}(W^\vee)$ is a cubic surface with an isolated $\mathrm{A}_2$-singularity
  if and only if
  $h\in(\igusa\cap P_i)\setminus \bigcup_{j\neq i} P_j$.
\end{lemma}

\begin{proof}
  A point $h\in \igusa \cap P_i$ induces a hyperplane section $H\cap \segre$
  such that $p_i\in H\cap\segre$ is a singularity worse than an $\mathrm{A}_1$-singularity.

  More precisely,
  the hyperplane $H$ will be at the node $p_i$ tangent to the smooth tangent cone of the node.
  That means that the tangent cone of the point $p_i$ considered in the hyperplane section $H\cap\segre$ is the cone over a smooth conic.
  We choose coordinates so that the node~$p_i$
  is $[1:0:0:0] \in P_i\cong \mathbb{P}_{x_0:x_1:x_2:x_3}^3$.
  Writing
  \begin{equation}
    f = \sum_{i=0}^3 x_0^i f_i(x_1,x_2,x_3)
  \end{equation}
  with $f_i$ homogenous of degree $i$,
  we have that $f_0=0$ as $p_i\in\segre$
  and the point $p_i$ being singular translates into $f_1=0$.
  The above discussion shows that $f_2$ is a singular conic
  and we can rewrite $f$ after projective transformation to have the form
  \begin{equation}
    f = x_0(x_1^2+x_2^2) + f_3(x_1, x_2,x_3)
  \end{equation}
  for $f_3$ a homogenous polynomial of degree $3$.

  The cubic surface $H\cap\segre$ will only be singular at the point $p_i$.
  Indeed, no tangent hyperplane at a point outside the Segre planes contains a node
  and tangent hyperplanes at smooth points on Segre planes lie on the singular locus of the Castelnuovo--Richmond quartic.
  Since $\igusa[\sing]\subset \bigcup_{j\neq i} P_j$
  we know that $H\cap\segre$ will neither be reducible nor contain any other node.
  This genericity implies that $f_3(x_1,x_2,x_3)$ defines a smooth elliptic curve inside the projective plane $x_0=0$
  and that the completion of the local ring of the surface $H\cap\segre$ at $p_i$
  is isomorphic to $\mathbb{C}[\![x_1,x_2,x_3]\!]/(x_1^2+x_2^2+x_3^3)$
  showing that the isolated singularity is an $\mathrm{A}_2$-singularity.

  For the converse, by \cref{lemma:one-nodal-hyperplane-segre}
  no tangent hyperplane at a point not lying on a Segre plane acquires a worse singularity than a node.
  Moreover, tangent hyperplanes at smooth points lying on a Segre plane produce reducible hyperplane sections.

  Thus, a hyperplane $H$ such that $H\cap\segre$ has an $\mathrm{A}_2$-singularity
  must contain a node $p_i$ and therefore $h$ must lie on some $P_i$.
  Moreover, since the singularity is worse than nodal,
  $h$ is contained in the intersection $P_i\cap \igusa$.

  To conclude, we need to show that $h\not \in P_j$ for $j\neq i$.
  This follows as the cubic surface is assumed to have a unique isolated singularity.
\end{proof}

\begin{remark}
  \label{remark:tangent-nodes-unique}
  In fact any hyperplane $H$ which is tangent at a node $p_i$
  is either nowhere else tangent and contains no other node
  or must contain a Segre plane.
  This follows from the above proof
  together with the fact that $\igusa\cap P_i\cap P_j$ is the union of two non-reduced lines,
  and therefore is contained in the singular locus of the Castelnuovo--Richmond quartic.
  For the latter claim, see \cite[Section 3.3.1]{MR1438547}.
\end{remark}
The small resolution of the Segre cubic produces a \emph{partial} resolution for the $\mathrm{A}_2$ singularity,
and we end up with a nodal singularity.

Since we are again considering a point $h$ in the smooth locus of the Castelnuovo--Richmond quartic, we have
\begin{equation}
  Y_h = \coble \times_{\mathbb{P}(W)} h = \Spec k[\epsilon]/(\epsilon^2).
\end{equation}

Next, we treat the case of several isolated singularities.

\begin{lemma}
  Let $r\in \{2,3,4\}$.
  A hyperplane section $H\cap \segre\subset\mathbb{P}(W^\vee)$ is an $r$-nodal cubic surface
  if and only if
  $h\in \mathbb{P}(W)$ lies on exactly $r$ different hyperplanes $P_i$'s,
  but not on the Castelnuovo--Richmond quartic.
  Such points $h\in \mathbb{P}(W)$ exist for all $r\in \{2,3,4\}$.
\end{lemma}

\begin{proof}
  Let us first prove the if direction.
  For a point $h \in \mathbb{P}(W)$ which is contained in exactly $r$ of the hyperplanes $P_i$
  the intersection $H\cap \segre$ will contain exactly $r$~nodes of the Segre cubic.
  Moreover, the intersection is transversal at all other points of the Segre cubic,
  since we assume that $h$ does not lie on the singular locus of the Castelnuovo--Richmond quartic
  and therefore does not contain a Segre plane
  (transversality outside the Segre planes again follows from \eqref{equation:projective-duality-isomorphism}).

  To prove the converse
  note that for a tangent hyperplane at a point of the Segre cubic
  we have two possibilities.
  Either the point lies on a Segre plane
  and the hyperplane section will therefore contain some $P_\sigma$.
  The second option is that the point does not lie on any of the Segre planes
  which by \cref{lemma:one-nodal-hyperplane-segre} is always one-nodal.
  \Cref{remark:tangent-nodes-unique} shows that such a hyperplane $h$ cannot lie on $\igusa$,
  since otherwise the hyperplane section would be a cone over an elliptic curve.

  It remains to show that such hyperplanes exist.
  For this we can consider the hyperplanes $T_{i,j}$
  which contain exactly four of the nodes
  and the intersection $T_{i,j} \cap \segre$
  is isomorphic to the Cayley cubic surface,
  the (unique up to isomorphism) cubic surface with four nodes.
  By what we have already proven,
  we know that for each hyperplane $T_{i,j}$
  there are exactly four distinct integers $a,b,c,d\in \{1,\dots ,10\}$ such that
  \begin{equation}
    t_{i,j} = P_a \cap P_b \cap P_c \cap P_d \in \mathbb{P}(W)\setminus \igusa.
  \end{equation}
  Thus a general point of the plane $P_a\cap P_b$
  corresponds to a hyperplane containing exactly two nodes
  and a general point of $P_a \cap P_b \cap P_c$
  yields a hyperplane containing exactly three of the nodes
  such that the hyperplanes do not contain any Segre plane.
\end{proof}

Since all nodes of the hyperplane section were already nodes on the Segre cubic,
they are resolved by the restriction of~$\rho$
and we again obtain a smooth weak del Pezzo surface.

On the dual side we have that~$h\notin\igusa$,
so that we always obtain that~$Y_h$ is 2~points.
Hence either by Orlov's blowup formula for~$X_H$ as an (iterated) blowup
or homological projective duality
we obtain~9 exceptional objects in the hyperplane section.

\begin{remark}
  The above cases of smooth, nodal and $\mathrm{A}_2$ singular cubic surfaces correspond exactly to hyperplanes $h\in \mathbb{P}(W)$ such that $h$ does not lie on the singular locus of the Castelnuovo--Richmond quartic.
\end{remark}

\begin{lemma}
  All other hyperplane sections~$H\cap\segre$ yield reducible cubic surfaces
  which are
  either the union of a plane and a quadric
  or the union of three Segre planes.
\end{lemma}

\begin{proof}
  It is immediate from the above discussion that we have already exhausted all possibilities for hyperplanes $H\subset \mathbb{P}(W^\vee)$
  such that the corresponding point $h\in \mathbb{P}(W)$
  does not lie on the singular locus of the Castelnuovo--Richmond quartic.
  This yields the assertion.
\end{proof}
We just remark that in the first case the residual quadric is the hyperplane section of the tangent cone of one of the nodes.

We will not provide an explicit description of $X_H$,
as this is a non-normal surface.
On the dual side something interesting happens with $Y_h$:
because~$h$ lies in the singular locus of~$\igusa$
we have that~$h\cap\coble$ is a double point,
\emph{but} $Y_h$ is not of the expected dimension.
By \cite[Lemma~2.35]{MR4076812} the fiber of~$Y\to\mathbb{P}(W)$ over~$h$
is either a line or a conic.
This means we have to compute~$Y_h$ as a derived fiber product.
We will illustrate such a computation for a hyperplane section of~$Y$ instead,
see \cref{lemma:tangent-hyperplane-igusa} and the ensuing discussion.

\begin{sidewaystable}
  \centering
  \begin{adjustbox}{center}
    \begin{tabular}{cc|c|cc}
      \toprule
      $H\cap\segre$                                 & $X_H$                        & $\mathcal{C}_H$                               & $Y_h$        & $h\cap\coble$ \\
      \midrule
      smooth cubic surface                          & smooth cubic surface         & 2 orthogonal objects                          & 2 points     & 2 points \\
      \addlinespace
      1-nodal cubic surface                         & singular cubic surface       & $\derived^\bounded(k[\epsilon]/(\epsilon^2))$ & double point & double point \\
                                                    & weak del Pezzo surface       & 2 orthogonal objects                          & 2 points     & 2 points \\
      \addlinespace
      2-nodal cubic surface                         & weak del Pezzo surface       & 2 orthogonal objects                          & 2 points     & 2 points \\
      \addlinespace
      3-nodal cubic surface                         & weak del Pezzo surface       & 2 orthogonal objects                          & 2 points     & 2 points \\
      \addlinespace
      4-nodal Cayley cubic surface                  & weak del Pezzo surface       & 2 orthogonal objects                          & 2 points     & 2 points \\
      \addlinespace
      cubic surface with $\mathrm{A}_2$ singularity & nodal weak del Pezzo surface & $\derived^\bounded(k[\epsilon]/(\epsilon^2))$ & double point & double point \\
      \addlinespace
      union of Segre plane and quadric              &                              &                                               &              & double point \\
      \addlinespace
      union of 3 Segre planes                       &                              &                                               &              & double point \\
      \bottomrule
    \end{tabular}
  \end{adjustbox}
  \caption{Description of all possible hyperplane sections of $X\to\mathbb{P}(W^\vee)$ and the homological projective dual sections}
  \label{table:segre-hyperplane}

  \bigskip
  \bigskip
  \bigskip

  \centering
  \begin{adjustbox}{center}
    \begin{tabular}{cc|c|cc}
      \toprule
      $H\cap\coble$                 & $Y_H$                                        & $\mathcal{C}_H$                               & $X_h$                                       & $h\cap\segre$ \\
      \midrule
      15-nodal quartic double solid & $\Bl_4\mathbb{P}(\mathrm{T}_{\mathbb{P}^2})$ & zero                                          & $\emptyset$                                 & $\emptyset$ \\
      \addlinespace
      non-reduced quadric           &                                              &                                               &                                             & point \\
      \addlinespace
      16-nodal quartic double solid & nodal weak Fano threefold                    & $\derived^\bounded(k[\epsilon]/(\epsilon^2))$ & $k[\epsilon]/(\epsilon^2)$, $|\epsilon|=-1$ & point \\
      \bottomrule
    \end{tabular}
  \end{adjustbox}
  \caption{Description of some interesting possible hyperplane sections of $Y\to\mathbb{P}(W)$ and the homological projective dual sections}
  \label{table:coble-hyperplane}
\end{sidewaystable}

\paragraph{Top-down approach}
It is also possible to directly consider hyperplane sections of~$X$
and fibers of~$\varpi\colon Y\to\mathbb{P}(W)$,
using the description as projective bundles.
Only \emph{a posteriori} do we make the link
to the more classical picture of the singular varieties~$\segre$ and~$\coble$.
We will only briefly explain this method,
to avoid too much redundancy with the earlier discussion.

The hyperplane section of~$\mathbb{P}(W^\vee)$
corresponding to~$L\subset W^\vee$ gives a surjective morphism~$W^\vee\to k$,
so that the fiber product~$X_L$ corresponding to the hyperplane section
\begin{equation}
  \begin{tikzcd}
    X_L \arrow[r, hook] \arrow[d] & X \arrow[d, "f"] \\
    \mathbb{P}(L) \arrow[r, hook] & \mathbb{P}(W^\vee)
  \end{tikzcd}
\end{equation}
can be written as~$\mathbb{P}_S(\mathcal{R})$,
where
\begin{equation}
  \mathcal{R}\colonequals\operatorname{im}(\mathcal{U}_2\to\mathcal{O}_S)
\end{equation}
using the composition
\begin{equation}
  \mathcal{U}_2\hookrightarrow W^\vee\otimes_k\mathcal{O}_S\twoheadrightarrow\mathcal{O}_S.
\end{equation}
We need to understand what the degeneracy locus of the morphism~$\mathcal{U}_2\to\mathcal{O}_S$ is.
The analysis is similar to that of \cite[Lemma~2.35]{MR4076812}:
the zero locus of a non-zero section~$\sigma\in\HH^0(\Gr(2,W^\vee),\mathcal{U}_2^\vee)$
is~$\Gr(2,4)$,
and from the description~$S=\Gr(2,W^\vee)\cap\mathbb{P}^5$
we see that there are 3 cases:
\begin{enumerate}
  \item a 0-dimensional scheme~$Z$ of length 2;
  \item a line~$L\subset S$;
  \item a conic~$C\subset S$.
\end{enumerate}
In the first case we get that~$\mathcal{R}\cong\mathcal{I}_Z$
and thus the composition
\begin{equation}
  \label{equation:sasha-composition}
  X_H\cong\mathbb{P}_S(\mathcal{I}_Z)\hookrightarrow X=\mathbb{P}_S(\mathcal{U}_2)\twoheadrightarrow S
\end{equation}
is equal to~$\operatorname{Bl}_ZS\to S$.
This can now be compared to the results in \cref{table:segre-hyperplane}:
when~$Z$ is reduced the position of the two points with respect to the~10~lines on~$S$ determines which case we are in,
when~$Z$ is non-reduced we are in the case that~$X_H$ is a nodal weak del Pezzo surface.

In the second case we have that~$\mathcal{R}\cong\mathcal{I}_L$,
and the restriction of~$\mathcal{U}_2$ to~$L$ is~$\mathcal{O}_L\oplus\mathcal{O}_L(-1)$,
as in the proof of \cite[Lemma~2.32]{MR4076812}.
We get that
\begin{equation}
  X_H\cong\mathbb{P}_L(\mathcal{U}_2|_L)\cup_L S\hookrightarrow X.
\end{equation}
The composition~$\mathbb{P}_L(\mathcal{U}_2|_L)\cong\mathbb{F}_1\hookrightarrow X\to\segre$ has as image a Segre plane
(with the restriction to this~$\mathbb{P}^2$ being the blowup),
and~$S$ gets blown down to the residual quadric of the hyperplane section through the Segre plane.

In the third case
the restriction of~$\mathcal{U}_2$ to~$C$ is now~$\mathcal{O}_C(-1)\oplus\mathcal{O}_C(-1)$,
and we get that
\begin{equation}
  X_H\cong\mathbb{P}_C(\mathcal{U}_2|_C)\cup_L S.
\end{equation}
If~$C$ is a smooth conic
we get that~$\mathbb{P}^1\times\mathbb{P}^1$ is mapped isomorphically onto a quadric in~$\segre$,
and~$S$ gets blown down to a Segre plane.

\subsection{Other linear sections}
\label{subsection:other-linear-sections}
We will now discuss hyperplane sections of the Coble fourfold,
and codimension~2~linear sections of both the Segre cubic and the Coble fourfold.
Doing so we can describe the derived categories of the linear sections of the resolutions.

\paragraph{Hyperplane sections of the Coble fourfold}
Dual to the discussion in \cref{subsection:hyperplane-segre}
we are considering hyperplane sections of the resolution~$Y$ of~$\coble\subset\mathbb{P}(W)$.
We will now write~$H\subset\mathbb{P}(W)$ for a hyperplane
corresponding to~$L\subset W$ of codimension~1,
and~$h\in\mathbb{P}(W^\vee)$ for the dual point~$\mathbb{P}(L^\perp)$.
Likewise~$X_h$ and~$Y_H$ denote the (derived) fiber products.
Observe that we are now applying \cref{section:hpd} to~$Y$
and consider~$X$ as its homological projective dual.

Because the singular locus of the Castelnuovo--Richmond quartic consists of~15~lines,
every hyperplane section~$H\subset\mathbb{P}(W)$ will necessarily intersect the singular locus.
An overview of the cases we will discuss
is given in \cref{table:coble-hyperplane}.

The generic case is when this happens in exactly~15~nodes,
so that we obtain a singular quartic surface with~15~nodes.
On the Segre side this corresponds to the hyperplane $h\in \mathbb{P}(W^\vee)$
not lying on the Segre cubic and not being contained in one of the hyperplanes $H_{i,j}$
the hyperplanes dual to the singular locus of the singular locus of the Castelnuovo--Richmond quartic.

The corresponding hyperplane section of~$\coble$ is a quartic double solid with~15~nodes,
branched along the singular quartic surface.
The following result shows how the Castelnuovo--Richmond quartic and Coble fourfold
are universal for such varieties \cite[Theorem~1 and~Proposition 2]{MR3954304},
\begin{proposition}[Avilov]
  \label{proposition:avilov}
  Let~$T$ be a quartic surface, singular in precisely~15~nodes.
  Then~$T$ is a hyperplane section of the Castelnuovo--Richmond quartic.
  Let~$Z$ be a quartic double solid, singular in precisely~15~nodes.
  Then~$Z$ is a hyperplane section of the Coble fourfold.
\end{proposition}
Incorporating the resolution~$\varpi$ into the picture,
the results of Avilov imply the following.
\begin{lemma}
  \label{lemma:quartic-double-solid}
  Let~$H\subseteq\mathbb{P}(W)$ be a hyperplane as in \cref{proposition:avilov}.
  Then~$Y_H$ is a smooth projective weak Fano threefold,
  obtained as a small resolution of a quartic double solid singular in~15~nodes.
  Its derived category has semiorthogonal decompositions
  \begin{equation}
    \begin{aligned}
      \derived^\bounded(Y_H)
      &=\langle\derived^\bounded(S),\derived^\bounded(S)\rangle \\
      &=\langle\derived^\bounded(\mathbb{P}(\mathrm{T}_{\mathbb{P}^2})),E_{1,1},E_{1,2},E_{2,1},E_{2,2},E_{3,1},E_{3,2},E_{4,1},E_{4,2}\rangle
    \end{aligned}
  \end{equation}
  where the~$E_{i,j}$ are exceptional objects.
\end{lemma}

\begin{proof}
  The hyperplane section~$Y_H$ is a smooth projective weak Fano threefold
  because it is a small resolution of a quartic double solid.
  The first semiorthogonal decomposition is induced by homological projective duality
  and consists of~$14=2\times 7$~exceptional objects,
  as~$X_h$ is empty in this case.
  The second semiorthogonal decomposition follows from an explicit description of the hyperplane section
  given in \cite[Proposition~1]{MR3954304},
  as the blowup of the Fano threefold~$\mathbb{P}(\mathrm{T}_{\mathbb{P}^2})$,
  isomorphic to a~$(1,1)$-section of~$\mathbb{P}^2\times\mathbb{P}^2$,
  in~4~points in general position,
  and applying Orlov's blowup formula.
  We again count~$14=6+4\times 2$ exceptional objects.
\end{proof}

\begin{remark}
  Observe that the description of~$Y_H$ in the proof of \cref{lemma:quartic-double-solid}
  parallels that of the \emph{second} small resolution of singularities of the Coble fourfold
  discussed in \cite[\S2.1]{MR4076812}.
  This small resolution is obtained as the blowup of~$\mathbb{P}^2\times\mathbb{P}^2$ in~4 points,
  and the small resolution in \cref{lemma:quartic-double-solid} is a hyperplane section of it.
\end{remark}

The next case we consider is that of a tangent hyperplane.

\begin{lemma}
  \label{lemma:tangent-hyperplane-igusa}
  For any point $p\in \igusa \cap P_i\setminus \igusa[\sing]$
  the tangent hyperplane is the non-reduced quadric $Q_i$.
  For any other point $p\in \igusa \setminus (\bigcup_i P_i)$
  the tangent hyperplane section $\tangent_p\igusa \cap \igusa$ is a singular Kummer quartic surface.
\end{lemma}

\begin{proof}
  The first statement is clear.
  For the second,
  use that the inverse of the duality map restricted to the complement of the hyperplanes $P_i$ is an isomorphism.
  This shows that the hyperplane will only be tangent at $p\in \igusa$.
  For more details, see \cite[Theorem~3.3.8]{MR1438547}.
\end{proof}


%

There are~16~singularities on the Kummer quartic surface~$H\cap\igusa$,
and~$H\cap\coble$ is a double cover ramified in the Kummer quartic,
a~16-nodal double quartic solid.
The restriction of the resolution~$Y\to\coble$ resolves all nodes except the one coming from the tangency point,
so that~$Y_H$ is a singular weak Fano fold.

To compute~$X_h$
we need to take the derived fiber product,
as~$h\in\segre\setminus\segre[\sing]$
implies that the usual fiber product~$X\times_{\mathbb{P}(W^\vee)}h$
is a single point,
which is not of the expected dimension~$-1$.
To compute the derived fiber product~$X_h$
we can take a Koszul resolution for~$\segre\subset\mathbb{P}(W^\vee)$:
\begin{equation}
  \label{equation:koszul}
  0\to\mathcal{O}_{\mathbb{P}(W^\vee)}(-3)\to\mathcal{O}_{\mathbb{P}(W^\vee)}\to\mathcal{O}_{\segre}\to0
\end{equation}
and consider the restriction of
the sheaf of dg~algebras~$[\mathcal{O}_{\mathbb{P}(W^\vee)}(-3)\to\mathcal{O}_{\mathbb{P}(W^\vee)}]$
living in degrees~$-1$ and~$0$
to~$h$.
This gives a dg~algebra~$A$
given by~$[k\to k]$
living in degrees~$-1$ and~$0$.
Because the morphism in the Koszul resolution is multiplication with the defining equation,
and~$h\in\segre$
we obtain that
the restricted differential vanishes.
Hence~$A=k[\epsilon]/(\epsilon^2)$ is the formal dg~algebra
where~$|\epsilon|=-1$.
The dg algebra~$A$ is an ingredient in the theory of absorption of nodal singularities
as introduced in \cite{kuznetsov-shinder-absorption} (see also \cite[Proposition 5.11 and \S5.3]{2111.00527v1}).

What is interesting in this case is that~$X_h$ has a derived structure,
but~$Y_H$ is a singular variety.
This frequently happens when~$X$ is (similar to) a closed subvariety.
In the next example \emph{both} fiber products will have a derived structure.

\paragraph{Plane sections of the Segre cubic}
Let~$L\subset W^\vee$ be a subspace of codimension~2.
We will write~$P=\mathbb{P}(L)\subset\mathbb{P}(W^\vee)$,
and~$p\subset\mathbb{P}(W)$ for the dual projective line.
We wish to describe~$X_P$ and~$Y_p$.

If~$P\subseteq\segre$ it is one of the fifteen Segre planes,
and~$X\times_{\mathbb{P}(W^\vee)}P$ is
either~$\mathbb{F}_1\cong\Bl_1\mathbb{P}^2$
or~$S\cong\Bl_4\mathbb{P}^2$.
This can be deduced from the arguments at the end of \cref{subsection:hyperplane-segre}
for hyperplane sections of~$\segre$
containing a Segre plane.
In any case it is not of the expected dimension.
On the dual side
$p$ is one of the~15~lines of the singular locus of~$\igusa$,
and~$Y\times_{\mathbb{P}(W)}p$ is described in \cite[Lemmas~2.41 and~2.43]{MR4076812}.
It is again not of the expected dimension.
In \eqref{equation:linear-sections-sod} there are no contributions from the Lefschetz center,
and we obtain an equivalence of categories
\begin{equation}
  \derived^\bounded(X_P)\cong\derived^\bounded(Y_p)
\end{equation}
where both~$X_P$ and~$Y_p$ are surfaces equipped with a derived structure.

If~$P$ is not strictly contained within the Segre cubic,
we obtain a plane cubic curve.
Let us first consider a plane cubic~$P\cap\segre$ contained in a smooth hyperplane section of~$\segre$
(as discussed in \cref{lemma:smooth-segre-hyperplane-criterion})
such that~$p\cap\igusa[\sing]=\emptyset$.
Then~$P\cap\segre$ is a reduced (but possibly reducible) plane cubic.
All cases except three lines meeting in one point
arise for a smooth cubic surface.
For this last case to occur the cubic surface needs to contain an Eckardt point.


Again we obtain an equivalence
\begin{equation}
  \derived^\bounded(X_P)\cong\derived^\bounded(Y_p),
\end{equation}
now for a (possibly singular) plane cubic
and a (possibly singular) double cover of~$\mathbb{P}^1$ ramified in a subscheme of length~4.
The reconstruction result \cite[Theorem~1.1]{MR3217412}
shows that we in fact have an isomorphism~$X_P\cong Y_p$.
We leave the non-generic situation to the interested reader.

%
%
%
%
%
\paragraph{Plane sections of the Coble fourfold}
Finally,
let~$L\subset W$ be a subspace of codimension~2.
We will write~$P=\mathbb{P}(L)\subset\mathbb{P}(W)$,
and~$p\subset\mathbb{P}(W^\vee)$ for the dual projective line.

There are~2~possible scenarios for~$p\cap\segre$:
\begin{itemize}
  \item $p\subset\segre$: there is a Fano surface of lines on~$\segre$
    described in \cite[\S4]{MR3776662},
    such that~$X_p$ acquires a derived structure
    and its geometry moreover depends on the position of~$p$ with respect to~$\segre[\sing]$;
  \item $\dim(p\cap\segre)=0$: the intersection is a scheme of length 3.
\end{itemize}
Let us consider the generic case of a zero-dimensional intersection.
Then~$X_p$ consists of~3~points
which avoid the singular locus of~$\segre$.
On the dual side we obtain a double cover of~$P=\mathbb{P}^2$,
ramified in a quartic~$P\cap\igusa$
which avoids~$\igusa[\sing]$,
so that~$Y_P$ is a (smooth) del Pezzo double plane,
a del Pezzo surface of degree~2.

By \eqref{equation:linear-sections-sod} we obtain semiorthogonal decompositions
\begin{equation}
  \label{equation:del-pezzo-double-plane}
  \begin{aligned}
    \derived^\bounded(X_p)&=\langle E_1,E_2,E_3\rangle \\
    \derived^\bounded(Y_p)&=\langle \derived^\bounded(X_p),\derived^\bounded(S)\rangle
  \end{aligned}
\end{equation}
where~$E_1,E_2,E_3$ are completely orthogonal.
Similar to \eqref{equation:sasha-composition}
we get that the composition~$Y_p\to S$
is the blowup in~3~points.
Generically it will give a del Pezzo surface as in \cref{table:generic},
but we will not perform the case-by-case analysis
of when the~3~points are not in general position
with respect to the~10~lines on~$S$.

If~$p\cap\segre[\sing]=\emptyset$ and~$X_p$ is still reduced
but~$P\cap\igusa[\sing]\neq\emptyset$
the fiber product~$Y_P$ is a smooth weak del Pezzo surface of degree~2,
and \eqref{equation:del-pezzo-double-plane} still holds.
If on the other hand~$p\cap\segre[\sing]=\emptyset$ and~$X_p$ non-reduced,
we obtain that~$Y_P$ is a singular weak del Pezzo surface of degree~2,
such that~$\derived^\bounded(X_p)$ in \eqref{equation:del-pezzo-double-plane}
becomes~$\derived^\bounded(\Spec k[\epsilon]/(\epsilon^2)\times k)$
or~$\derived^\bounded(\Spec k[\epsilon]/(\epsilon^3))$.
We leave it to the interested reader to match this up with the classification of \cite[\S8.7.1]{MR2964027}

If~$p\cap\segre[\sing]\neq\emptyset$ the fiber product~$X_p$ acquires a derived structure
and we will not discuss it further.

\section{The Segre cubic vs.~moduli of quiver representations}
\label{section:segre-quiver}
As explained in \cref{section:segre-igusa-coble} the Segre cubic has a modular interpretation
as the Satake compactification of a Picard modular variety.
This interpretation a priori does not yield anything interesting for the resolution~$X$
from the point of view of homological projective duality.
But there is a different modular interpretation of~$\segre$ and the resolution~$X$,
which gives rise to a \emph{second} rectangular Lefschetz decomposition.

By \cite[page~17]{MR1007155} and \cite[\S5.1]{MR4352662} we have that~$\segre$ is
the moduli space of semistable quiver representations for the~6\dash subspace quiver
\begin{equation}
  \label{equation:6-subspace}
  \subspace_6\colon
  \begin{tikzpicture}[baseline,vertex/.style={draw, circle, inner sep=0pt, text width=2mm}]
    \node[vertex] (a) at (0:.8cm)   {};
    \node[vertex] (b) at (60:.8cm)  {};
    \node[vertex] (c) at (120:.8cm) {};
    \node[vertex] (d) at (180:.8cm) {};
    \node[vertex] (e) at (240:.8cm) {};
    \node[vertex] (f) at (300:.8cm) {};

    \node[vertex] (g) at (0,0) {};

    \draw[-{Classical TikZ Rightarrow[]}] (a) -- (g);
    \draw[-{Classical TikZ Rightarrow[]}] (b) -- (g);
    \draw[-{Classical TikZ Rightarrow[]}] (c) -- (g);
    \draw[-{Classical TikZ Rightarrow[]}] (d) -- (g);
    \draw[-{Classical TikZ Rightarrow[]}] (e) -- (g);
    \draw[-{Classical TikZ Rightarrow[]}] (f) -- (g);
  \end{tikzpicture}
\end{equation}
where we use the dimension vector~$(1,1,1,1,1,1;2)$,
and the canonical stability condition as discussed in \cite[\S2.2]{MR4352662}.
This stability condition lies on a wall,
and by considering a small perturbation of the stability condition
we can obtain small desingularisations \cite[Theorem~4.3]{MR3644807}.
Two particular choices were studied in \cite{MR3778138},
and by the description of their automorphisms we know that they correspond to types~IV and~VI in \cite{MR0914085},
hence by \cref{remark:small-resolutions-segre}
we know that they are isomorphic to~$\mathbb{P}_S(\mathcal{U}_2)$ and~$\Bl_5\mathbb{P}^3$.

\paragraph{A fully faithful functor in modular settings}
On the other hand,
in full generality there is
an expected relationship between~$\derived^\bounded(kQ)$ and~$\derived^\bounded(M)$,
where~$Q$ is an acyclic quiver and~$M$ is a suitable moduli space of semistable representations.
Namely we expect that,
for the right choice of dimension vector and stability condition,
there is a fully faithful embedding of~$\derived^\bounded(kQ)$ into~$\derived^\bounded(M)$
given by the universal representation.

There exists a rich literature on similar admissible embeddings
into derived categories of moduli spaces:
\begin{itemize}
  \item for curves of~$g\geq 2$ and moduli spaces of vector bundles \cite{MR3764066,MR3713871,MR3954042,2106.04857v1};
  \item for Hilbert schemes of points \cite{MR3397451,MR3950704}
\end{itemize}
In the case of noncommutative algebra, we refer to
\begin{itemize}
  \item for two noncommutative surfaces and Hilbert schemes of points \cite{MR4165471,MR3488782};
  \item for quivers, provided the dimension vector is thin (i.e.~the moduli space is toric) \cite{MR1688469}.
\end{itemize}

\paragraph{Mutating the original Lefschetz structure}
Starting from the Lefschetz structure in \cref{theorem:main-theorem-precise}
we can perform a sequence of mutations to find a new Lefschetz structure,
whose existence is related to the expectation outlined above.

%
%
%
%
\begin{proposition}
  \label{proposition:3-block-collection}
  Let~$S$ be the del Pezzo surface of degree~5.
  There exists a 3-block exceptional collection
  \begin{equation}
    \label{equation:3-block}
    \derived^\bounded(S)=
    \langle
      \mathcal{O}_S;
      \mathcal{O}_S(h-e_1),
      \mathcal{O}_S(h-e_2),
      \mathcal{O}_S(h-e_3),
      \mathcal{O}_S(h-e_4),
      \mathcal{O}_S(2h-e);
      \mathcal{U}_2^\vee
    \rangle
  \end{equation}
  where~$h$ is the pullback of the hyperplane class of~$\mathbb{P}^2$,
  $e_1,\ldots,e_4$ are the classes of the exceptional divisors,
  we set~$e\colonequals e_1+e_2+e_3+e_4$,
  and~$\mathcal{U}_2$ is as in \cref{subsection:resolutions}.
  The~5~line bundles in the middle block correspond to the~5~conic bundle structures on~$S$.

  Moreover, we have isomorphisms
  \begin{equation}
    \label{equation:ext-isomorphisms}
    \Ext_S^\bullet(\mathcal{O}_S(h-e_i),\mathcal{U}_2^\vee)\cong k[0],\qquad\Ext_S^\bullet(\mathcal{O}_S(2h-e),\mathcal{U}_2^\vee)\cong k[0].
  \end{equation}
\end{proposition}

\begin{proof}
  We obtain this collection by mutating the 3-block exceptional collection from \cite[page~452]{MR1642152}.
  Their collection, written using the notation of the statement of the proposition, is
  \begin{equation}
    \label{equation:karpov-nogin-3-block}
    \derived^\bounded(S)=
    \langle
      \mathcal{O}_S;
      \mathcal{F};
      \mathcal{O}_S(h),
      \mathcal{O}_S(e_1-\omega_S-h),
      \mathcal{O}_S(e_2-\omega_S-h),
      \mathcal{O}_S(e_3-\omega_S-h),
      \mathcal{O}_S(e_4-\omega_S-h)
    \rangle
  \end{equation}
  where~$\mathcal{F}$ is the vector bundle obtained as the universal extension
  \begin{equation}
    \label{equation:definition-of-F}
    0\to\mathcal{O}_S(-\omega_S-h)\to\mathcal{F}\to\mathcal{O}_S(h)\to 0.
  \end{equation}
  We have that~$\mathcal{F}\cong\mathcal{U}_2^\vee$.
  Indeed,
  the Chern classes of $\mathcal{F}$ and $\mathcal{U}_2^\vee$ satisfy
  \begin{equation}
    \mathrm{c}_1(\mathcal{F})= -\omega_S = \mathrm{c}_1(\mathcal{U}_2^\vee),
    \quad\mathrm{c}_2(\mathcal{F}) = -\omega_Sh-h^2=2=\mathrm{c}_2(\mathcal{U}_2^\vee).
  \end{equation}
  The last equality follows from classical Schubert calculus saying that $\sigma_1^4\sigma_{1,1}=2$ on~$\Gr(2,5)$.
  Since both bundles are of rank two and exceptional
  (for~$\mathcal{U}_2^\vee$ this is a Koszul computation)
  they must be isomorphic by \cite[Proposition~1.3]{MR1642152}.

  Mutating the third block of the last~5~line bundles in \eqref{equation:karpov-nogin-3-block}
  to the very left corresponds to tensoring the objects with~$\omega_S$.
  Now perform a right mutation on the first two blocks.
  The resulting mutated exceptional sheaves are computed using \emph{division},
  i.e.~they are the cokernels in the short exact sequences
  \begin{equation}
    0\to\mathcal{O}_S(\omega_S+h)\to\Hom_S(\mathcal{O}_S(\omega_S+h),\mathcal{O}_S)^\vee\otimes\mathcal{O}_S\to\mathcal{O}_S(2h-e)\to 0
  \end{equation}
  and
  \begin{equation}
    0\to\mathcal{O}_S(e_i-h)\to\Hom_S(\mathcal{O}_S(e_i-h),\mathcal{O}_S)^\vee\otimes\mathcal{O}_S\to\mathcal{O}_S(h-e_i)\to 0.
  \end{equation}
  The Hom-spaces in the middle term are~2\dash dimensional,
  as they are identified with the global sections of the~5~conic bundle structures on~$S$,
  and the identification of the cokernels is a Chern class computation.
  The result is the 3-block exceptional collection in \eqref{equation:3-block}.

  The identifications in \eqref{equation:ext-isomorphisms} follow from \eqref{equation:definition-of-F},
  so that we obtain
  \begin{equation}
    \begin{aligned}
      \Ext_S^\bullet(\mathcal{O}_S(h-e_i),\mathcal{U}_2^\vee)&\cong\HH^\bullet(S,\mathcal{O}_S(e_i)) \\
      \Ext_S^\bullet(\mathcal{O}_S(2h-e),\mathcal{U}_2^\vee)&\cong\HH^\bullet(S,\mathcal{O}_S).
    \end{aligned}
  \end{equation}
  by the vanishing of~$\HH^\bullet(S,\mathcal{O}_S(h-e+e_i))$ and~$\HH^\bullet(S,\mathcal{O}_S(e-h))$ thanks to the projection formula.
\end{proof}

%

From Orlov's projective bundle formula applied to~$p\colon X\to S$
we therefore obtain the semiorthogonal decomposition
\begin{equation}
  \label{equation:orlov}
  \begin{aligned}
    &\derived^\bounded(X)=
    \langle
      \mathcal{O}_X;
      \mathcal{O}_X(h-e_i)\mid i=1,2,3,4;
      \mathcal{O}_X(2h-e),
      p^*\mathcal{U}_2^\vee; \\
      &\quad
      \mathcal{O}_X(s);
      \mathcal{O}_X(s+h-e_i)\mid i=1,2,3,4;
      \mathcal{O}_X(s+2h-e),
      p^*\mathcal{U}_2^\vee(s)
    \rangle
  \end{aligned}
\end{equation}
where we denote the relative hyperplane class by~$s$.
This is of course nothing but the decomposition \eqref{equation:lefschetz-for-segre-cubic}
written as a full exceptional collection using \eqref{equation:3-block}.
We can now modify this collection into a \emph{new} Lefschetz decomposition,
with a Lefschetz center that is not equivalent to the Lefschetz center in \eqref{equation:lefschetz-for-segre-cubic}.

%
%
%
%
%
%
%
%

\begin{proposition}
  \label{proposition:quiver-lefschetz-structure}
  There exists a rectangular Lefschetz decomposition
  \begin{equation}
    \derived^\bounded(X)=
    \langle\mathcal{A}_0,\mathcal{A}_1(1)\rangle
  \end{equation}
  with respect to the line bundle~$\mathcal{O}_X(1)=\mathcal{O}_X(s)=f^*\mathcal{O}_{\mathbb{P}(W^\vee)}(1)$,
  where the Lefschetz center is
  \begin{equation}
    \label{equation:quiver-lefschetz-center}
    \mathcal{A}_0=
    \langle
      \mathcal{O}_X(h-e_i)\mid i=1,\ldots,4;
      \mathcal{O}_X(2h-e),
      \mathcal{O}_X(3h-e-s),
      p^*\mathcal{U}_2^\vee
    \rangle
  \end{equation}
  such that~$\mathcal{A}_0=\mathcal{A}_1\cong\derived^\bounded(k\subspace_6)$
  is the derived category of the 6-subspace quiver from \eqref{equation:6-subspace}.
\end{proposition}

\begin{proof}
  The right mutation of~$\mathcal{O}_X$ with respect to its orthogonal complement
  gives the anticanonical line bundle~$\mathcal{O}_X(2s)$ as the final object.
  We obtain a rectangular Lefschetz decomposition
  with Lefschetz center
  \begin{equation}
    \langle
      \mathcal{O}_X(h-e_i)\mid i=1,\ldots,4;
      \mathcal{O}_X(2h-e),
      p^*\mathcal{U}_2^\vee,
      \mathcal{O}_X(s)
    \rangle
  \end{equation}
  To identify this with \eqref{equation:quiver-lefschetz-center}
  and show that~$\mathcal{A}_0\cong\derived^\bounded(k\subspace_6)$
  we do the left mutation on the objects~$p^*\mathcal{U}_2^\vee$
  and~$\mathcal{O}_X(s)$
  which does not change the Lefschetz center.
  The left mutation is defined by the kernel in the short exact sequence
  \begin{equation}
    0\to\mathrm{L}_{p^*\mathcal{U}_2^\vee}\mathcal{O}_X(s)\to\Hom_X(p^*\mathcal{U}_2^\vee,\mathcal{O}_X(s))\otimes p^*\mathcal{U}_2^\vee\to\mathcal{O}_X(s)\to 0,
  \end{equation}
  using the identification~$\mathcal{U}_2^\vee\cong p_*\mathcal{O}_X(s)$
  and the exceptionality of~$\mathcal{U}_2^\vee$,
  which gives the surjectivity
  and the kernel is a line bundle by a rank computation.
  The identification with~$\mathcal{O}_X(3h-e-s)$
  follows from a Chern class computation using \eqref{equation:definition-of-F}.

  We can now check that~$\mathcal{A}_0\cong\derived^\bounded(k\subspace_6)$.
  We have that
  \begin{equation}
    \Ext_X^\bullet(\mathcal{O}_X(h-e_i),p^*\mathcal{U}_2^\vee)\cong\Ext_X^\bullet(\mathcal{O}_X(2h-e),p^*\mathcal{U}_2^\vee)\cong k[0]
  \end{equation}
  by the last part of \cref{proposition:3-block-collection}.
  The first 6~objects are completely orthogonal:
  the first 5~because they originate from a block in the 3-block collection,
  the orthogonality with the 6th object follows from the vanishing of~$\RRR p_*\mathcal{O}_X(-s)$ in one direction
  and from the exceptional sequence in the other.
\end{proof}

\paragraph{Comparison of Lefschetz structures}
We have now obtained~3~(rectangular) Lefschetz structures on~$\derived^\bounded(X)$:
\begin{enumerate}
  \item the projective bundle Lefschetz structure used in \cref{theorem:main-theorem-precise},
    where the Lefschetz center is~$\derived^\bounded(S)$;

  \item the blowup Lefschetz structure from \eqref{equation:blowup-lefschetz-structure} in \cref{proposition:blowup-lefschetz-structure},
    where the Lefschetz center is~$\derived^\bounded(A)$
    for the finite-dimensional algebra~$A=kQ/I$ where~$Q$ is the quiver
    
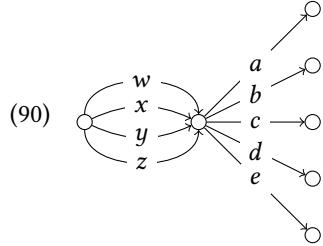
\begin{equation}
      \begin{tikzpicture}[scale=.75,baseline,vertex/.style={draw, circle, inner sep=0pt, text width=2mm}]
        \node[vertex] (a) at (-1,0)  {};
        \node[vertex] (b) at (1,0)  {};
        \node[vertex] (c) at (3,2)  {};
        \node[vertex] (d) at (3,1)  {};
        \node[vertex] (e) at (3,0)  {};
        \node[vertex] (f) at (3,-1) {};
        \node[vertex] (g) at (3,-2) {};

        \draw[-{Classical TikZ Rightarrow[]}] (a) to [bend left=90]  node[fill=white] {$w$} (b);
        \draw[-{Classical TikZ Rightarrow[]}] (a) to [bend left=25]  node[fill=white] {$x$} (b);
        \draw[-{Classical TikZ Rightarrow[]}] (a) to [bend right=25] node[fill=white] {$y$} (b);
        \draw[-{Classical TikZ Rightarrow[]}] (a) to [bend right=90] node[fill=white] {$z$} (b);
        \draw[-{Classical TikZ Rightarrow[]}] (b) to node[fill=white] {$a$} (c);
        \draw[-{Classical TikZ Rightarrow[]}] (b) to node[fill=white] {$b$} (d);
        \draw[-{Classical TikZ Rightarrow[]}] (b) to node[fill=white] {$c$} (e);
        \draw[-{Classical TikZ Rightarrow[]}] (b) to node[fill=white] {$d$} (f);
        \draw[-{Classical TikZ Rightarrow[]}] (b) to node[fill=white] {$e$} (g);
      \end{tikzpicture}
    \end{equation}
    and~$I$ is the ideal of relations
    \begin{equation}
      I=(xa, ya, za,
        wb, yb, zb,
        wc, xc, zc,
        wd, xd, yd,
        we-xe, xe-ye, ye-ze)
    \end{equation}
    which encodes~5~points in general position on~$\mathbb{P}_{w:x:y:z}^3$
    as up to the action of~$\mathrm{PGL}_4$ we can take these to be
    \begin{equation}
      (1:0:0:0),(0:1:0:0),(0:0:1:0),(0:0:0:1),(1:1:1:1),
    \end{equation}
    so that the structure follows from the composition law in \eqref{equation:blowup-lefschetz-structure};

  \item the quiver Lefschetz structure from \cref{proposition:quiver-lefschetz-structure},
    where the Lefschetz center is~$\derived^\bounded(k\subspace_6)$.
\end{enumerate}

The following propositions gives a comparison between these Lefschetz structures,
in the sense of \cite[Definition~6.9]{MR3948688}.
We have that
\begin{itemize}
  \item the blowup and quiver Lefschetz structures agree;
  \item the projective bundle Lefschetz structure is different from the other two.
\end{itemize}

\begin{proposition}
  \label{proposition:lefschetz-comparisons}
  The projective bundle Lefschetz structure
  is not Lefschetz equivalent
  to the quiver Lefschetz structure.
\end{proposition}

\begin{proof}
  We observe that the Lefschetz centers are already inequivalent.
  Indeed,~$\derived(S)$ is not equivalent to~$\derived^\bounded(k\subspace_6)$,
  as~$\HHHH^2(S)\cong\HH^0(S,\omega_S^\vee)=k^6$ by the Hochschild--Kostant--Rosenberg theorem,
  whilst~$\HHHH^2(k\subspace_6)=0$ as~$k\subspace_6$ is hereditary
  (and as~$\subspace_6$ is a tree we moreover have~$\HHHH^1(k\subspace_6)=0$).
\end{proof}

Hence we have found a genuinely new Lefschetz decomposition for~$\derived^\bounded(X)$.
On the other hand we have the following identification.

\begin{proposition}
  \label{proposition:michel}
  The blowup Lefschetz center
  is equivalent
  to the quiver Lefschetz center.
\end{proposition}

\begin{proof}
  We want to prove that~$\derived^\bounded(A)\cong\derived^\bounded(k\subspace_6)$.
  One can compute (e.g.~using \cite{qpa} and the explicit presentation given above)
  that~$\HHHH^1(A)=\HHHH^2(A)=0$,
  and~$\HHHH^{\geq3}(A)=0$ by global dimension reasons.
  Now consider~$A$ as a one-point extension of the~5\dash quotient quiver
  by the representation~$M$.
  By the long exact Hochschild cohomology sequence~\cite[Theorem~5.3]{MR1035222}
  we see that~$M$ is an exceptional representation.

  On the other hand~$k\subspace_6$ can also be seen as a one-point extension,
  of the~5\dash subspace quiver,
  using the indecomposable projective (thus exceptional) representation~$P$ concentrated at the sink.

  Let~$M'$ be the image of~$P$ under the composition of
  the Nakayama functor (sending it to the indecomposable injective at the sink)
  and the reflection functor at the sink.
  This is also an exceptional representation for the~5\dash quotient quiver,
  so~$M\cong M'$ as we have~$\operatorname{\mathbf{dim}}M=\operatorname{\mathbf{dim}}M'=(1,1,1,1,1;4)$
  if we put the source as the last vertex.

  By considering the derived versions of the Nakayama functor and the reflection functor
  we get that their composition is an equivalence of derived categories for the~5\dash subspace and~5\dash quotient quiver,
  which sends~$P$ to~$M$.
  Hence by \cite[Theorem~1]{MR1980681} we obtain an induced derived equivalence between one-point extensions.
  This gives an identification of the Lefschetz centers.
\end{proof}

%

It would be interesting to further understand (noncommutative) homological projective duality
for this second Lefschetz center.

\renewcommand*{\bibfont}{\normalfont\small}
\printbibliography

\small

\emph{Pieter Belmans}, \texttt{pieter.belmans@uni.lu} \\
Universit\'e du Luxembourg, Avenue de la Fonte 6, L-4364 Esch-sur-Alzette, Luxembourg

\emph{Thorsten Beckmann}, \texttt{beckmann@math.uni-bonn.de} \\
Max--Planck Institut für Mathematik, Vivatsgasse 7, 53111 Bonn, Germany

\end{document}